\documentclass[12pt,a4paper,reqno]{amsart}
\usepackage[T1]{fontenc}
\usepackage{lmodern}
\usepackage{amsmath,amsthm,amssymb,amsfonts,mathtools,mathrsfs}
\usepackage[utf8]{inputenc}
\usepackage{tikz}
\usepackage{tikz-cd}
\usetikzlibrary{braids,knots}
\usepackage[shortlabels]{enumitem}
\setlist[itemize]{label=\textbullet}

\usepackage{listings}
\usepackage{csquotes}
\usepackage[titletoc]{appendix}
\usepackage[colorlinks,linkcolor=black,urlcolor=blue,citecolor=black]{hyperref}

\usepackage[top=3.2cm, bottom=3cm, left=3.5cm, right=3.5cm]{geometry}
\usepackage{cellspace}
\usepackage{adjustbox}
\usepackage{algorithm}   

\theoremstyle{definition}
\newtheorem{definition}{Definition}

\theoremstyle{definition}

\theoremstyle{remark}

\theoremstyle{remark}
\newtheorem{Remark}{Remark}

\theoremstyle{remark}
\newtheorem*{remark}{Remark}

\theoremstyle{remark}

\theoremstyle{remark}

\theoremstyle{remark}

\theoremstyle{plain}
\newtheorem{Theorem}{Theorem}

\theoremstyle{plain}

\newcommand{\thistheoremname}{}
\newtheorem*{genericthm}{\thistheoremname}
\newenvironment{namedthm}[1]
  {\renewcommand{\thistheoremname}{#1}%
   \begin{genericthm}}
  {\end{genericthm}}

\theoremstyle{plain}

\theoremstyle{plain}
\newtheorem{Corollary}[Theorem]{Corollary}

\theoremstyle{plain}
\newtheorem{Lemma}[Theorem]{Lemma}

\theoremstyle{plain}

\theoremstyle{plain}
\newtheorem{Proposition}{Proposition}

\theoremstyle{plain}

\newcommand*{\N}{\mathbb{N}}
\newcommand*{\Z}{\mathbb{Z}}
\newcommand*{\Q}{\mathbb{Q}}
\newcommand*{\R}{\mathbb{R}}
\newcommand*{\Cc}{\mathbb{C}}
\newcommand*{\h}{\mathbb{H}}
\newcommand*{\F}{\mathbb{F}}

\newcommand*{\A}{\mathbb{A}}
\newcommand*{\Gg}{\mathbb{G}}
\usepackage{mathtools}
\DeclareMathOperator{\Hom}{Hom}
\DeclareMathOperator{\GL}{GL}
\DeclareMathOperator{\GU}{GU}
\DeclareMathOperator{\U}{U}
\DeclareMathOperator{\SU}{SU}
\DeclareMathOperator{\SL}{SL}

\title{A $\Lambda$-adic Kudla lift}
\author{Francesco Maria Iudica}
\date{}
\address{Normandie Univ, UNICAEN, CNRS, LMNO, 14000 Caen, France}
\email{francesco-maria.iudica@unicaen.fr}
\begin{document}
\maketitle
\begin{abstract}
\noindent The Kudla lift studied in this article is a classical version for Picard modular forms of the automorphic theta lift between $\GU(2)$ and $\GU(3)$. We construct an explicit $p$-adic analytic family of Picard modular forms varying with respect to the weight and level, which interpolates a so-called $p$-modification of the lift at arithmetic weights, by exploiting a formula of Finis for the Fourier-Jacobi coefficients of a lifted form.
\end{abstract}
{\let\clearpage\relax \tableofcontents} 
\thispagestyle{empty}
\setcounter{section}{-1}
\section{Introduction}
In his seminal paper \cite{kudlaq}, Kudla defined a lift $\mathcal{L}$ from elliptic modular forms to holomorphic modular forms for unitary groups of signature $(m,1)$. He proved that this lifting is in fact arithmetic, in the sense that if $f$ is a modular form on the upper half-plane with respect to a congruence subgroup of $\SL_2(\Z)$ and whose Fourier coefficients are in $\Q^{\text{ab}}$, then there exists a complex period $\Omega$ independent of $f$ such that $\Omega^{-1}\mathcal{L}(f)$ is arithmetic in the sense of Shimura \cite[p. 580]{shimura}. Other authors, such as Gelbart-Rogawski \cite{gelbart}, and Murase-Sugano \cite{murase}, have studied the theta lift for the dual reductive pair
\[
(\GU(1,1), \GU(2,1))
\] 
from an automorphic point of view.
This theta lift is an extension of the Kudla lift for $m=2$ to the space of automorphic forms for $\GU(1,1)$.


In his unpublished Ph.D. thesis \cite{finis}, Finis  further studied the arithmeticity of the classical level $1$ Kudla lift
\begin{align*}
\mathcal{L}_k:M_{k-1}(\Gamma_1(D),\omega_{K/\Q})&\to\mathscr{M}(k, L, \chi)\\
f&\mapsto \mathcal{L}_k(f)
\end{align*}
for $k\geq 6$ an even integer, obtaining in particular an explicit formula for the Fourier-Jacobi coefficients of the lifted form $\mathcal{L}_k(f)$ in terms of $f$. 
In this article we adress a problem first proposed by Finis in his thesis, \cite[p. 6]{finis}: for $p\geq 5$ a prime which splits in a fixed imaginary quadratic number field $K/\Q$, we construct $p$-adic analytic families of Picard modular forms varying in weight and level, which interpolate the classical Kudla lifts of Hida families of elliptic cusp forms with Nebentypus. In order to do so, we first recall the definition of the level $1$ Kudla lift.
As we shall see, the Fourier-Jacobi coefficients obtained by Finis, which are so-called generalized theta functions in the terminology of Shimura \cite{shimura}, do not vary $p$-adically. This leads to our first result, namely the construction of a suitable modification of the above lift, which we call the $p$-modified lift:
\[
\mathcal{L}^p_{k}:S_{k-1}(\Gamma_1(p^rD), \omega_{K/\Q})\to\mathscr{S}(k, \Gamma, \chi),
\]
for a certain congruence subgroup $\Gamma\subset \GU(2,1)(\Q)$. Here and throughout this article, $S_{k-1}(\Gamma_1(p^rD), \omega_{K/\Q})$ is the space of modular forms of weight $k-1$, level $\Gamma_1(p^rD)$ and Nebentypus $\omega_{K/\Q}$ at $D$. In particular, any Nebentypus at $p$ is allowed. The above procedure can be thought of as a sort of $p$-stabilization of the level $1$ Kudla lift. Indeed, the effect of this modification is that of adding $p$ in the level, and of removing any $p$-power from the coefficients. However, we note here that we have not been able to compare this procedure with the usual $p$-stabilization of eigenforms. Having removed the obstruction to the interpolation of the level $1$ Kudla lift, we get the following result (see Theorem \ref{main} in the text for a more general version):
\begin{namedthm}{Main Theorem}
Given a Hida family $\mathfrak{F}(\kappa)$ with ordinary specializations $$\mathfrak{F}(k)=f_k\in S_{k}^{\emph{ord}}(\Gamma_1(pD),\omega_{K/\Q})$$ for $k\in \Z$, there exists a $\Lambda$-adic Fourier-Jacobi expansion defining a $\Lambda$-adic Picard modular form $\mathfrak{L}(\mathfrak{F})(\kappa)$, such that
\[
\Omega_0^{-k}\mathfrak{L}(\mathfrak{F})(k)=\mathcal{L}^p_{k}(f_{k-1}),
\]
where $\Omega_0$ is a period only depending on $K$.
\end{namedthm}

It is noteworthy that both the level $1$ and the $p$-modified lifts are never ordinary in the sense of Hida. Indeed, for each rational prime $q$, the eigenvalue attached to $q$ of any eigenform obtained by this lift will not be invertible modulo $q$. 

As an application of these results, we mention a possible interpolation of the adjoint Kudla lift. Indeed, applying the results of Loeffler \cite{loeffler} to our setting, we obtain, in a suitable sense, $p$-adic families of the special cycles defined in \cite{cogdell}. It makes thus sense to ask whether the $p$-adic variation of these cycles is compatible with the interpolation of their image under the adjoint Kudla lift, which is an elliptic cusp form in a Hida family. This will be pursued in future work.

The paper is organized as follows: in \S$1$-$2$ we introduce the necessary notations and recall the definition of Picard modular surfaces. In \S$3$ we define Picard modular forms following the formulation of Finis, and we introduce the notion of arithmetic theta functions. The description of the theta lifts, of level $1$ and $p$ respectively, occupies \S$4$. The main result of the paper, Theorem \ref{main}, is in the last section \S$5$.

\section{Notation}
Let $K$ be an imaginary quadratic number field and fix an inclusion $\sigma:K\hookrightarrow \Cc$, and let $\overline{\sigma}:K\hookrightarrow \Cc$ be its complex conjugate. We use the following notation:
\begin{itemize}
    \item $d$ is the squarefree integer such that $K=\Q(\sqrt{d})$, $d<0$,
    \item $D$ is the discriminant of $K$, that is
    \[
D=\begin{cases}
d & d\equiv 1\!\!\mod 4\\
4d & d\equiv 2,3\!\! \mod 4,
\end{cases}    
    \]
    \item $\delta=\sqrt{D}$ is the square root of the discriminant with positive imaginary part,
    \item $\mathcal{O}_{K}=\Z+\Z\tau_K$, so that 
    \[
\tau_K=\begin{cases}
\frac{1+\sqrt{d}}{2} & \text{if } D\; \text{is odd,}\\
\sqrt{d} & \text{otherwise}.
\end{cases}    
    \]
\end{itemize}
\noindent Throughout the article, we fix an odd prime $p$ split in $K$, so that $p\mathcal{O}_K=\mathfrak{p}\overline{\mathfrak{p}}$. From now on, we fix $\iota_\infty:\overline{\Q}\hookrightarrow \Cc$ an embedding of an algebraic closure of $\Q$ in $\Cc$ such that $\iota_\infty\mid_K=\sigma$, and we identify the image of $\iota_\infty$ with the image under a fixed embedding $\iota_{p}:\overline{\Q}\hookrightarrow\overline{\Q}_p$ once and for all.

\noindent Let $\A_\Q$ denote the ring of adeles of $\Q$ and $\A_{\Q,f}$ the ring of finite adeles. The ideal class group of $K$ is denoted by $\text{Cl}_{K}$. We write $w_{K}$ for the number of units in $K$. Write $\A_K= \A_\Q\otimes_\Q K$ and $\A_{K,f}= \A_{\Q,f}\otimes_\Q K$. For any finite set $S$ of places of $\Q$, write $\A_{\Q}^S$ for the adeles away from $S$. We will pass freely from finite adeles to fractional ideals of $K$, and thus we will use the same notation for both objects. 

\subsection{Conventions on Hecke characters}
We give here a description of the conventions used throughout the article for Hecke characters associated to the imaginary quadratic field $K=\Q(\sqrt{d})$. If $\mathfrak{f}$ is a non-zero ideal of $\mathcal{O}_K$ and $I_K^{S(\mathfrak{f})}$ is the group of fractional ideals of $K$ coprime to $\mathfrak{f}$, a classical Hecke character is a group homomorphism $\varphi:I_K^{S(\mathfrak{f})}\to \Cc^\times$ satisfying
\[
\varphi(\alpha\mathcal{O}_K)=\varphi_{\text{fin}}(\alpha)\alpha^a\overline{\alpha}^b
\]
for all $\alpha\in \mathcal{O}_K$ that are coprime to $\mathfrak{f}$, where $\varphi_{\text{fin}}:(\mathcal{O}_K/\mathfrak{f})^\times\to \Cc^\times$ is a character, called the finite part of $\varphi$, and $a,b$ are complex numbers. We call the pair $(a,b)$ the infinity type of $\varphi$, and $\mathfrak{f}$ the conductor of $\varphi$. In the case an algebraic character has infinity type of the form $(k/2,-k/2)$ for some integer $k$, we just say that it has weight $k$. We say that a character is unramified if it has trivial conductor, i.e. $\mathfrak{f}=\mathcal{O}_K$. With this convention, the norm character defined by $\text{N}(\mathfrak{q})=|\mathcal{O}_K/\mathfrak{q}|$ on prime ideals is unramified, has trivial finite part, and infinity type $(1,1)$.
It is sometimes useful to view $\varphi$ as being defined on all fractional ideals of $\mathcal{O}_K$, setting $\varphi(\mathfrak{a})=0$ if $\mathfrak{a}$ is not coprime to $\mathfrak{f}$. Recall that classical Hecke characters are in bijection with adelic Hecke characters, which are defined as continuous homomorphisms $\A_K^\times/K^\times\to \Cc^\times$. A straightforward way to describe this bijection is by letting 
\begin{align*}
\text{ide}: \A_{K,f}^\times&\to I_K\\
\alpha&\mapsto\prod\mathfrak{q}^{v_{\mathfrak{q}}(\alpha_\mathfrak{q})}
\end{align*} 
the map that associates to a finite idele $\alpha=(\alpha_v)$, for $v$ the places of $K$, its corresponding fractional ideal; then a classical Hecke character $\varphi:I_K^{S(\mathfrak{f})}\to \Cc^\times$ corresponds to a continuous character $\varphi_\Cc:\A_K^\times/K^\times\to \Cc^\times$ such that $\varphi_\Cc(\alpha)=\varphi(\text{ide}(\alpha))$ for every $\alpha\in \A_{K,f}^\times$ which is trivial at the places of $S(\mathfrak{f})$. Under this correspondence, we find that if $\varphi$ has infinity type $(a,b)$ then the local factor $\varphi_{\Cc,\infty}$ at the infinity place is given by $\varphi_{\Cc,\infty}(z)=z^{-a}\overline{z}^{-b}$. Thus our convention for the infinity type for an adelic Hecke character is that it corresponds to the negative of the exponents of $z$ and $\overline{z}$ in the local component at the infinity place. In particular, the adelic absolute value $||\cdot||_{\A_K}$ is a character of infinity type $(-1,-1)$, and it corresponds to the inverse of the norm character.
 
A classical or adelic Hecke character of $K$ is called algebraic if its infinity type $(a,b)$ consists of integers. Algebraic adelic Hecke characters are in bijection with algebraic $p$-adic Hecke characters, which we briefly describe. A $p$-adic Hecke character is a continuous homomorphism $\psi:\A_K^\times/K^\times\to \overline{\Q}_p^\times$, where the domain is endowed with the idelic topology, and $\psi$ is algebraic of integral weights $(a,b)$ if its local factors $\psi_{\mathfrak{p}}$ and 
$\psi_{\overline{\mathfrak{p}}}$ on $K_{\mathfrak{p}}\cong\Q_p^\times$ and $K_{\overline{\mathfrak{p}}}\cong \Q_p^\times$ are given by $\psi_{\mathfrak{p}}(x)=x^{-a}$ and $\psi_{\overline{\mathfrak{p}}}(x)=x^{-b}$ on some neighborhoods of the identity in these multiplicative groups of local fields. Then, an algebraic adelic Hecke character $\varphi_\Cc$ of infinity type $(a,b)$ corresponds to a $p$-adic Hecke character $\varphi_{\Q_p}$ of weights $(a,b)$ by the formula
\[
\varphi_{\Q_p}(\alpha)=(\iota_p\circ \iota_\infty^{-1})(\varphi_{\Cc}(\alpha)\alpha^{a}_{\infty}\overline{\alpha}^b_\infty)\alpha^{-a}_{\mathfrak{p}}\alpha^{-b}_{\overline{\mathfrak{p}}}
\]
for any idele $\alpha=(\alpha_v)$. It is easy to see that this defines a continuous character $\A_K^\times\to \overline{\Q}_p^\times$, and it is trivial on $K^\times$ because if $\alpha\in K^\times$ is treated as a principal idele, then $\iota_\infty^{-1}(\alpha_\infty)=\iota_p^{-1}(\alpha_{\mathfrak{p}})$ and $\iota_\infty^{-1}(\overline{\alpha}_\infty)=\iota_p^{-1}(\alpha_{\overline{\mathfrak{p}}})$, because of our convention on embeddings.

\section{The unitary Shimura variety of signature $(2,1)$}
Endow the $K$-vector space $V=K^3$ with the hermitian pairing
\begin{center}
    $(u,v)=\overline{u}^tJ v$,
\end{center}
for 
\[
J=\begin{pmatrix}& & \delta^{-1}\\
    & 1 & \\
    -\delta^{-1} & &\end{pmatrix}.
\]
We identify $V_\Cc=V\otimes_{K} \Cc$ with $\Cc^3$, letting $K$ act on it via $\sigma$. The space $V_\Cc$ is hermitian of signature $(2,1)$. 
\begin{definition}
Let $G:=\GU(V,(,))$ be the general unitary group of $V$. For any ring $R$, define the $R$-points of the group scheme $G$ as,
\[
G(R):=\{(g,\mu(g))\in \GL_3(R\otimes\mathcal{O}_K) \times R^\times \mid \overline{g}{^t}Jg =\mu(g) J\}.
\]
\end{definition}
\noindent The map $\mu:G\to \mathbb{G}_m$ is a character of algebraic groups, and is called the similitude factor. Note that taking determinants of $\overline{g}{^t}Jg=\mu(g)J$ we see that 
\[
\mu(g)^3=\det \overline{g}\det g=|\det g|^2.
\]
We define $\U:=\ker \mu$. As $\mu$ is uniquely determined by $g$, we will often identify the couple $(g,\mu(g))$ with $g$. We let $B$ denote the Borel subgroup of all upper triangular matrices in $G$, which is also the stabilizer of the point $(1:0:0)$ in $\mathbb{P}^2(K)$. Its unipotent radical is denoted by $N$. We write elements of $N$ as $[v,u]$ with $v\in \text{Res}_{K/\Q}\Gg_{\text{a}}$ and $u\in \Gg_{\text{a}}$, where
\[
[v,u]=\begin{pmatrix}
1&\delta\overline{v}&u+\delta v\overline{v}/2\\
0&1&v\\
0&0&1
\end{pmatrix}.
\]
The center of $N$ is the group of all elements $[0,u]$ and is therefore isomorphic to $\Gg_{\text{a}}$.

The group $G_\infty=G(\R)$ acts on $\mathbb{P}^2_\Cc=\mathbb{P}(V_\Cc)$ by projective linear transformations and preserves the open subdomain $\mathfrak{X}$ of negative definite lines, which is biholomorphic to the open unit ball in $\Cc^2$. Every negative definite line is represented by a unique vector $(z,w,1)^t$ and such a vector represents a negative line if and only if
\begin{center}
    $\eta(z,w):= (z-\overline{z})/\delta-w\overline{w}> 0$.
\end{center}
For a point $\mathfrak{z}=(z,w)\in \mathfrak{X}$, we denote $\underline{\mathfrak{z}}=(z,w,1)^t$ the corresponding negative vector in $V$.
The upper half-plane $\h$ embeds in $\mathfrak{X}$ as the set of points with $w=0$.
Fix the point $\mathfrak{z}_0=(\delta/2,0)\in \mathfrak{X}$, and let $K_\infty$ be its stabilizer in $G_\infty$. The latter is compact modulo center: $K_\infty/Z(G_\infty)\cong K_\infty \cap \U(\R)$ is compact and isomorphic to $\U(2)(\R)\times \U(1)(\R)$. Since $G_\infty$ acts transitively on $\mathfrak{X}$, we may identify the latter with 
\[
G_\infty/K_\infty\cong \U(2,1)(\R)/(\U(2)(\R)\times \U(1)(\R)).
\] 
The pair $(G,\mathfrak{X})$ is a Shimura datum in the sense of Deligne, with reflex field $K$. Let $L$ be a full $\mathcal{O}_{K}$-lattice in $V$, such that $L$ is included in
\[
L^\vee:=\{v\in V\mid \text{tr}((u,v))\in \Z,\; \forall\, u\in L\},
\]
the dual of $L$. Let $G(L)_f\subset G(\mathbb{A}_f)$ the subgroup stabilizing $\hat{L}:=L\otimes \hat{\Z}$. An adelic level subgroup is an open compact subgroup $\mathcal{U}_f\subseteq G(L)_f$. We will suppose that $\mathcal{U}_f=\mathcal{U}^p\mathcal{U}_p$ with $\mathcal{U}^p\subseteq G(\mathbb{A}^{p,\infty})$ and $\mathcal{U}_p\subseteq G(\Z_p)$. In the following, we will only be interested in the case where $\mathcal{U}_p$ is hyperspecial, i.e. $\mathcal{U}_p=G(\Z_p)$, or parahoric. More precisely, fix the canonical basis $\{e_1,e_2,e_3\}$ for $\F_p^3$. Call $P_1\subset \GL_3(\Z_p)$ (resp. $P_2\subset \GL_3(\Z_p)$), the inverse images via the reduction map $\GL_3(\Z_p)\to \GL_3(\F_p)$ of the subgroups of $\GL_3(\F_p)$ that stabilize the line $\langle e_1\rangle$ (resp. the plane $\langle e_1,e_2\rangle$). 
Since $p$ splits in $K$, under the isomorphism $G(\Z_p)\cong \mathbb{G}_m(\Z_p)\times \GL_3(\Z_p)$, the parahoric level structure that we consider corresponds to the case $\mathcal{U}_p=\mathbb{G}_m(\Z_p)\times P_2$. 

\noindent For each choice of adelic level subgroup $\mathcal{U}_f\subseteq G(L)_f$, define $\Gamma_{\mathcal{U}_f}:=\mathcal{U}_f\cap G(\Q)$, the associated classical level subgroup. Call the conductor of $\mathcal{U}_f$, or of $\Gamma_{\mathcal{U}_f}$ indistinctly, the squarefree integer which is the product of the primes $\ell$ at which $\mathcal{U}_\ell\subsetneq G(\Z_\ell)$.
The double coset space
\[
S_{\mathcal{U}_f}(\Cc):= G(\Q)\backslash G(\A_\Q)/\mathcal{U}_fK_\infty=G(\Q)\backslash (\mathfrak{X}\times G(\A_{\Q,f}))/\mathcal{U}_f,
\]
is in fact the set of complex points of a quasi-projective variety $S_{\mathcal{U}_f}/\Cc$ of dimension two, the Shimura variety associated to the reductive group $G$, of level $\mathcal{U}_f$. The character $\nu:\det/\mu:G\to T$, where $T$ is the torus $\text{Res}_{K/\Q}\Gg_m$, induces a map
\[
\nu:S_{\mathcal{U}_f}(\Cc)\to T(\A_\Q)/T(\Q)\nu(Z(\R)\mathcal{U}_f)
\]
from $S_{\mathcal{U}_f}(\Cc)$ to a finite group, whose fibers are the connected components of $S_{\mathcal{U}_f}(\Cc)$. In particular, for $\mathcal{U}_f=G(L)_f$, the Shimura surface $S$ of level $1$ has $h_K$ connected components, cf. \cite[p. 13]{goren}. The principal congruence subgroup $\mathcal{U}_f(N)$ of $G(L)_f$ is the normal subgroup obtained as the kernel of the reduction modulo $N$ map on $G(L)_f$, i.e. the subgroup operating trivially on $\hat{L}/N\hat{L}$.

The Shimura varieties obtained this way allow an interpretation as moduli spaces for polarized abelian varieties with additional endomorphism and level structures. To sketch this interpretation, let $V=K^3$, $L=\mathcal{O}_K^3, J$, be as above. given a point $\mathfrak{z}\in \mathfrak{X}$, we obtain a splitting $V_\Cc=V_+\oplus V_-$ as a sum of a positive two-dimensional space $V_+$ and a negative one-dimensional space $V_-$. This splitting defines a canonical complex structure $j_\mathfrak{z}$ on $V_\Cc$: for $v=v_++v_-\in V_+\oplus V_-$ and $z\in \Cc$, set
\[
j_\mathfrak{z}(z)(v_++v_-):=zv_++\overline{z}v_-.
\]
Additionally, the alternating $\R$-linear form $E=\text{Tr}_{\Cc/\R}\delta(\cdot,\cdot)$ has the properties $E(j_\mathfrak{z}(i)x,j_\mathfrak{z}(i)y)=E(x,y)$ and $E(j_\mathfrak{z}(i)x,x)<0$ for $x\in V_\Cc\setminus \{0\}$. If therefore $\Lambda\subset V$ is on $\mathcal{O}_K$-lattice such that the form $E$ is integral on $\Lambda$, we obtain an abelian threefold $V_\Cc/\Lambda$, with complex structure $j_\mathfrak{z}$ and polarization given by $E$, together with an $\mathcal{O}_K$-multiplication, such that the Rosati involution induces the non-trivial automorphism of $\mathcal{O}_K$.

We now associate to a pair $(\mathfrak{z},g_f)\in \mathfrak{X}\times G(\A_{\Q_f})$ the lattice $\Lambda=g_fL$ and change the hermitian form by the factor $|\mu(g_f)|_{\A_{\Q}}$. Then, our construction gives a polarized abelian threefold $A$ with additional $\mathcal{O}_K$-structure. Multiplication of $g_f$ by an element of $G(L)_f$ on the right does not change anything, and multiplication of $(\mathfrak{z}, g_f)$ by an element of $G(\Q)$ from the left induces an isomorphism of the two structures. Consequently, we have an interpretation, over $\Cc$, of the above surface of level $1$ as a moduli space. To get a corresponding interpretation for level $N$ Picard variety, we must add a level structure, i.e. a trivialization of the $N$-torsion of $A$, cf. \cite[p. 11]{goren}.
\subsection{PEL type moduli spaces}
Since $S_{\mathcal{U}_f}(\Cc)$ parametrizes abelian varieties with additional structure over $\Cc$, it is natural to give it an arithmetic structure by showing the representability of the corresponding moduli functor problem over a number ring.

Fix a prime ideal over $p$, say $\mathfrak{p}$. For each level subgroup $\mathcal{U}_f=\mathcal{U}^p\mathcal{U}_p\subset G(\A_{\Q,f})$, we define the following moduli problem that sends each connected scheme $T$ over $Spec(\mathcal{O}_{K,\mathfrak{p}})$ to the set of tuples $(A,\lambda, \eta_{\mathcal{U}_f},i)$ up to isomorphism, where:
\begin{enumerate}[(i)]
\item $A\to T$ is an abelian scheme of relative dimension $3$.
\item $i:\mathcal{O}_K\to \text{End}(A)\otimes_\Z \Z_{(p)}$ is an injective morphism of $\Z_{(p)}$-algebras.
\item $\lambda:A\to A^\vee$ is a principal polarization such that the Rosati involution induces complex conjugation on $i(\mathcal{O}_K)$
\item $\eta_{\mathcal{U}_f}$ is a $\mathcal{U}_f$-level structure.
\end{enumerate}
We explain what the level structure $\eta_{\mathcal{U}_f}$ means. For an abelian variety $A$ of dimension $3$ over $T$ satisfying the first three conditions above, and a geometric point $x\in T$, we can look at the trivialization of the Tate module
\[
H_1(A_x,\A_{\Q,f}^p)\cong V\otimes\A_{\Q,f}^p.
\]
This is an isomorphism of symplectic modules, with the symplectic structure on the left hand side given by the Weil pairing and on the right hand side by $E$. The group $G(\A_{\Q,f}^p)$ acts on $V\otimes\A_{\Q,f}^p$ and so does $\mathcal{U}^p\subset G(\A_{\Q,f}^p)$. We are interested in two cases:
\begin{enumerate}[(i)]
\item If $\mathcal{U}_p=G(\Z_p)$, then $\eta_{\mathcal{U}_f}$ means a choice of a $\mathcal{U}^p$-orbit of the above isomorphism.
\item If $\mathcal{U}_p=\Z_p^\times\times P_2$, then $\eta_{\mathcal{U}_f}$ means a choice of a $\mathcal{U}^p$-orbit of the above isomorphism, plus a choice of a subgroup $H$ of $A[\mathfrak{p}]$ of order $p^2$.
\end{enumerate}
The following is well-known and can be found in \cite[I.4.1.11]{lan}:
\begin{Theorem}
If the level $\mathcal{U}^p$ is neat, the above moduli problem is represented by a smooth, quasi-projective scheme $S_{\mathcal{U}_f}$ over $\mathcal{O}_{K,\mathfrak{p}}$.
\end{Theorem}
Above $S_{\mathcal{U}_f}$ there is a universal abelian scheme $\mathcal{A}$ of dimension $3$ with an $\mathcal{O}_K$-action. We also have the relative differential sheaf $\omega_{\mathcal{A}}:=e^*\Omega^1_{\mathcal{A}/S_{\mathcal{U}_f}}$, where $e:S_{\mathcal{U}_f}\to \mathcal{A}$ is the identity section. This is a locally free sheaf of rank $3$ which inherits the $\mathcal{O}_K$-action
of $\mathcal{A}$. We have a decomposition $\omega_{\mathcal{A}} = \omega_{\mathfrak{p}} \oplus \omega_{\overline{\mathfrak{p}}}$.
In addition, by our assumption on the signature of $V$, the type of $\omega_{\mathcal{A}}$ as an $\mathcal{O}_{S_{\mathcal{U}_f}}$-module is also $(2, 1)$. Thus $\text{rank}_{\mathcal{O}_{S_{\mathcal{U}_f}}}
(\omega_{\mathfrak{p}}) = 2$ and $\text{rank}_{\mathcal{O}_{S_{\mathcal{U}_f}}}
(\omega_{\overline{\mathfrak{p}}}) = 1$
on all connected components of $S_{\mathcal{U}_f}$.
\section{Picard modular forms}
\label{q-expansion}
Recall that $G=\GU(2,1)$, $G_\infty=G(\R)$ and that we have fixed a point $\mathfrak{z}_0\in \mathfrak{X}$.
In order to define scalar-valued modular forms for $G$, first of all we note that function 
\[
j(g,\mathfrak{z})=g_{31}z+g_{32}w+g_{33},\;\; g=(g_{ij})\in G_\infty, \mathfrak{z}\in\mathfrak{X}
\]
is a holomorphic factor of automorphy for $G$, cf. \cite[p. 22]{goren}.
We give the classical analytic definition of Picard modular forms.

\begin{definition}[Picard modular form]
\begin{enumerate}[(i)]
\item A scalar-valued modular form $F$ for $G$ of weight $k\in \Z$ and level $\mathcal{U}_f\subset G(L)_f$, is a complex valued function on the double coset space $$G(\Q)\backslash G(\A_{\Q})/\mathcal{U}_f\to \Cc,$$ which satisfies the following: for each $g_f\in G(\A_f)$ the function
\[
F_{g_f}(g_\infty):=j(g_\infty,\mathfrak{z}_0)^{k}\mu(g_\infty)^{-k/2}F(g_\infty g_f)
\]
of $g_\infty\in G_\infty$ depends only on $\mathfrak{z}=g_\infty \mathfrak{z}_0$ and is holomorphic in $\mathfrak{z}$. We denote by $\mathscr{M}(k,\Gamma_{\mathcal{U}_f})$ the space of all scalar-valued modular forms of weight $k$ and level $\Gamma_{\mathcal{U}_f}$.
\item For $\chi$ a unitary Hecke character of $K$, let $\mathscr{M}(k,\Gamma,\chi)$ be the space of functions $F\in\mathscr{M}(k,\Gamma)$ with central character $\chi$, i.e.
\[
F(zg)=\chi(z)F(g)\;\;\; \forall\, z\in Z(\Q)\backslash Z(\A_\Q), g\in G(\A_\Q),
\]
for $Z$ the center of $G$.
\item A modular form $F\in \mathscr{M}(k,\Gamma)$ is called a cusp form if for every $g_f\in G(\A_{\Q,f})$, the function $F_{g_f}$ goes to zero for $\mathfrak{z}=(z,w)$, $z\to i\infty$. The subspace of cusp forms is denoted $\mathscr{S}(k,\Gamma)$. For $\Gamma=\Gamma(L)$ the stabilizer of the lattice $L$, we will simply write $\mathscr{M}(k,L)$ and $\mathscr{S}(k,L)$.
\end{enumerate}
\end{definition}

\begin{Remark}
For the space $\mathscr{M}(k,\Gamma,\chi)$ to be non-zero, $\chi$ needs to have infinity type $(k/2,-k/2)$.
\end{Remark}
We give two other equivalent definitions of Picard modular forms: one of a geometric nature, and the other algebraic. Let $R$ be an $\mathcal{O}_K[1/2D]$-algebra. A geometric modular form of weight $k \geq 0$
and neat level $\mathcal{U}_f$ defined over $R$ is an element of the finite $R$-module
$$\mathscr{M}^{\text{geom}}(k,\mathcal{U}_f,R) := H^0(S_{\mathcal{U}_f,R},\omega_{\overline{\mathfrak{p}}}^k).$$
A cusp form is an element of the space
$$\mathscr{S}^{\text{geom}}(k,\mathcal{U}_f,R) := H^0(S_{\mathcal{U}_f,R},\omega_{\overline{\mathfrak{p}}}^k\otimes \mathcal{O}(C)^\vee),$$
where $C$ is the boundary divisor.
An algebraic Picard modular form à la Katz of weight $k$ and level
$\mathcal{U}_f$ defined over $R$, is a rule $f$ which assigns to every $R$-scheme $T$, and every
$\underline{A}=(A,\lambda, \eta_{\mathcal{U}_f},i) \in S_{\mathcal{U}_f,R}(T)$, together with a nowhere vanishing section $\omega \in H^0(T, \omega_{A/T} (\overline{\sigma}))$, an element $f(\underline{A}, \omega) \in H^0(T,\mathcal{O}_T )$ satisfying
\begin{itemize}
\item $f(\underline{A}, \lambda\omega) = \lambda^{-k}f(A, \omega)$ for every $\lambda \in H^0(T,\mathcal{O}_T )^\times$;
\item The rule $f$ is compatible with base change $T'/T$.
\end{itemize}
We denote by $\mathscr{M}^{\text{alg}}(k,\mathcal{U}_f,R)$ the space of all algebraic Picard modular forms defined over $R$, of weight $k$, level $\mathcal{U}_f$.
\begin{Theorem}
We have a bijection between $\mathscr{M}^{\text{geom}}(k,\mathcal{U}_f,\Cc)$, $\mathscr{M}^{\text{alg}}(k,\mathcal{U}_f,\Cc)$ and $\mathscr{M}(k,\mathcal{U}_f,\Cc)$.
\end{Theorem}
\begin{proof}
See the discussion in \cite[p. 25]{goren}.
\end{proof}
\subsection{The Fourier-Jacobi expansion}
In this section, we follow closely Finis's presentation. Let $\lambda$ be the additive character of $\A_{\Q}/\Q$ normalized by $\lambda(x_\infty)=e^{2\pi ix_\infty}$. The Fourier-Jacobi coefficients $F_r(g), r\in \Q$, of a form $F\in \mathscr{M}(k,\Gamma,\chi)$ are given by integration against the center of $N$:
\[
F_r(g)=\int_{\A_{\Q}/\Q}F((0,u)g)\lambda(-ru)du,\;\;\; F(g)=\sum_{r\in \Q}F_r(g).
\]
Following \cite[p. 570]{shimura}, we may define the space of generalized theta functions $T_{r,\mathfrak{b}}$, for $\mathfrak{b}$ a fractional ideal of $\mathcal{O}_{K}$ and an element $r\in \Q^{\geq 0}$, such that $r\text{N}(\mathfrak{b})$ is integral, as the space of holomorphic functions $\vartheta:\Cc\to \Cc$ satisfying the functional equation
\[
\vartheta(w+s)=\psi(s)e^{-2\pi ir\delta\overline{s}(w+s/2)}\vartheta(w),\;\; s\in \mathfrak{b},
\]
where $\psi(s):=(-1)^{rD|s|^2}$. 
The classical Fourier-Jacobi expansion may be extracted out of the $F_r(g)$ as follows: for $w_\infty\in \Cc$ and $a\in \A_{K}^\times$, define
\[
g_{r,a}(w_\infty)=a_\infty^{-k}F_r((w_\infty,0)a)e^{-\pi i r\delta(|a_\infty|^2+|w_\infty|^2)}.
\]
This depends only $w_\infty$ and on $a_f\in \mathbb{A}_{K,f}^\times$, the finite part of $a$,  and is holomorphic in $w_\infty$. Then, for $\mathfrak{a}$ the fractional ideal of $K$ such that $\text{ide}(a_f)=\mathfrak{a}$, $g_{r,a}\in T_{r,\mathfrak{a}}$ whenever $r\text{N}(\mathfrak{a})$ is integral, and vanishes otherwise. Moreover, replacing $a$ by $\lambda a$ with $\lambda\in K^\times$, we have the relation
\[
g_{r/\text{N}(\lambda),\lambda a}(\lambda w)=\lambda^{-k}g_{r,a}(w).
\]
We will therefore write $g_{r,\mathfrak{a}}:=g_{r,a}$.
Fourier-Jacobi coefficients are therefore parametrized by pairs $(\mathfrak{a},d)$ consisting of a fractional ideal $\mathfrak{a}$ of $\mathcal{O}_{K}$, and a non-negative integer $d=r\text{N}(\mathfrak{a})$. We may define the space $\mathcal{T}_{d,k}$ of Fourier-Jacobi coefficients of degree $d$ as the space of all vectors $(t_{\mathfrak{a}})\in \prod_{\mathfrak{a}\in I_{K}}T_{d/\text{N}(\mathfrak{a}),\mathfrak{a}}$ such that
\[
t_{\lambda\mathfrak{a}}(\lambda w)=\lambda^{-k}t_{\mathfrak{a}}(w).
\] 
Obviously, we always have that $g_d:=(g_{d/\text{N}(\mathfrak{a}),\mathfrak{a}})\in \mathcal{T}_{d,k}$. After choosing a system of representatives $\mathcal{A}$ for the ideal classes of $K$ we get an isomorphism
\[
\mathcal{T}_{d,k}\cong \bigoplus_{\mathfrak{a}\in \mathcal{A}}T_{d/\text{N}(\mathfrak{a}),\mathfrak{a},k}^1,
\]
where $T_{r,\mathfrak{a},k}\subset T_{r,\mathfrak{a}}$ denotes the subspace of functions that satisfy the above weight $k$ functional equation, and $T_{r,\mathfrak{a},k}^1\subset T_{r,\mathfrak{a},k}$ denotes the subspace of theta functions $\vartheta$ invariant under the action of the roots of unity.
The coefficient $g_{0,\mathfrak{a}}$ is the constant term of $F$ at the cusp corresponding to $\mathfrak{a}$. It is an element of the one-dimensional space $T_{0,\mathfrak{a},k}$, which consists of constant theta functions. Consequently, $\mathcal{T}_{0,k}$ has dimension $h_{K}=|\text{Cl}_{K}|$.

\subsection{Arithmetic theta functions}
\label{arithmetic}
The notion of an arithmetic automorphic form for $G$ has been introduced by Shimura in \cite[p. 580]{shimura}. An element in $\mathscr{M}(k,\Gamma)$ is said to be arithmetic if its Fourier-Jacobi coefficients are arithmetic theta functions:
\begin{definition}
\begin{enumerate}[(i)]
\item A theta function $g_d\in \mathcal{T}_{d,k}$ is called arithmetic if $$e^{-\pi d\sqrt{D}|w|^2}g_d(w)\in\iota_\infty(\mathcal{O}_{K^{\emph{ab}}}),\; \forall\, w\in K.$$ Denote by $\mathcal{T}_{d,k}^{\emph{ar}}\subset \mathcal{T}_{d,k}$ the subspace of arithmetic theta functions.
\item Denote by $\mathscr{M}^{\emph{ar}}(k,\Gamma)$ (resp. $\mathscr{S}^{\emph{ar}}(k,\Gamma)$) the space of arithmetic Picard modular forms (resp. cusp forms).
\end{enumerate}
\end{definition}

\noindent The notions introduced in Shimura's are in analogy to the case of elliptic modular forms, where the $q$-expansion principle holds: the $q$-expansion homomorphism is injective, and if the $q$-expansion of a modular form $f$ is defined over a ring $S$, then the modular form is already defined over $S$. In the $\GU(2,1)$ case, the existence of a good compactified moduli scheme (away from primes dividing $D$) is known from work of Larsen \cite{larsen1},\cite{larsen2}, see also \cite{bellaiche}. This, together with the discussion in section \ref{q-expansion}, shows that we can read the ring of definition of a Picard modular form from its $q$-expansion, which justifies the above definitions.
\section{The $p$-modified Kudla lift}
\label{kudla}
In this section, we first give an account of the level one Kudla lift, and then introduce the $p$-modified lift, which is a modification of the particular Kudla lift considered by Finis, in order to make this compatible with the desired $p$-adic interpolation.

\subsection{The level one Kudla lift}
\label{theta}
Throughout this subsection, let $L:=\mathcal{O}_{K}^3$ be the standard lattice.
For $k$ a positive even integer, we consider the theta lift
\[
\mathcal{L}:S_{k-1}(\Gamma_1(D),\omega_{K/\Q})\to\mathscr{S}(k, L, \chi),
\]
where we recall that
\[
\Gamma_1(D)=\bigg\{\gamma\in \SL_2(\Z)\Big| \gamma\equiv \begin{pmatrix}
1&*\\
0&1
\end{pmatrix}\!\!\!\!\mod D\bigg\},
\]
and $\omega_{K/\Q}$ is the non-trivial quadratic character associated to $K/\Q$. We also take $\chi$ to be an unramified unitary Hecke character with infinity component $\chi_{\infty}(z)=(z/|z|)^{-k}$, corresponding to infinity type $(k/2,-k/2)$, according to our conventions ; we will say $\chi$ has weight $k$. Additionally, let $\varepsilon$ be an unramified Hecke character of weight $0$, i.e. a character of $\text{Cl}_{K}$. We define the theta kernel $$\theta(\tau,\mathfrak{z})=\theta_{k,\varepsilon,\chi}(\tau,\mathfrak{z}):\h\times G(\A_\Q)\to \Cc$$ by
\[
\overline{\theta(\tau, g)}=\sum_{\mathfrak{b}\in \text{Cl}_{K}}\varepsilon(\det g)|\text{N}(\mathfrak{b})\mu(g_f)|_{\A_{\Q}}^{k/2}(\varepsilon^3\chi^{-1})(\mathfrak{b})\overline{\theta_{\mathfrak{b}}(\tau, g)}
\]
with 
\[
\overline{\theta_{\mathfrak{b}}(\tau, g)}=2^{k-1}D^{-k/2}y^2\eta(\mathfrak{z}_0)^{-k}\mu(g_\infty)^{-k/2}\sum_{X\in \mathfrak{b} g_fL}\overline{f_{\tau,\mathfrak{z}}^{\mathfrak{b}g_f}(X)}.
\]
where
\[
\overline{f_{\tau,\mathfrak{z}}^{\mathfrak{b}g_f}(X)}=(X,g_\infty \underline{\mathfrak{z}_0})^ke^{2\pi i|N(\mathfrak{b})\mu(g_f)|_{\A_{\Q}} (2iy|(X,g_\infty \underline{\mathfrak{z}_0})|^2(\mu(g_\infty)\eta(\mathfrak{z}_0))^{-1}+\overline{\tau}(X,X))}.
\]
Here $\tau=x+iy\in \h$, and recall that $\mathfrak{z}_0:=(\delta/2,0)\in \mathfrak{X}$.
The following proposition is due to Kudla:
\begin{Theorem}\emph{\cite[Theorem 5.3]{kudlaq}}
\begin{enumerate}[(i)]
\item For a fixed $g_0\in G(\mathbb{A}_\Q)$ the function $\theta(\tau,g_0):\h\to \Cc$ is automorphic, i.e. for $\gamma\in \Gamma_0(D)$, we have:
\[
\theta(\gamma\tau,g_0)=\omega_{K/\Q}(\gamma)j(\gamma,\tau)^{k-1}\theta(\tau,g_0).
\]
Moreover, it is rapidly decreasing in $\tau$.
\item For a fixed $\tau_0\in h$, $\theta(\tau_0,g):G(\Q)\backslash G(\A_\Q)/G(L)_f\to \Cc$ is a smooth function.
\item Define the lifting $\mathcal{L}:=\mathcal{L}_{k,\varepsilon,\chi}$ by
\[
\mathcal{L}(f)(g)=\int_{\Gamma_1(D)\backslash \h}f(\tau)\overline{\theta(\tau,g)}y^{k-3}dxdy
\]
for $f\in M_{k-1}(\Gamma_1(D),\omega_{K/\Q})$. Then $\mathcal{L}(f)$ is an element of $\mathscr{M}(k,L,\chi)$. If $f$ is cuspidal, $\mathcal{L}(f)$ is also cuspidal.
\end{enumerate}
\end{Theorem}

We now define the intrinsic theta functions associated to $K$, in the sense of Shimura \cite[\S 4]{shimura2}. For a fractional ideal $\mathfrak{a}$ of $K$ set
\[
\vartheta_\mathfrak{a}(w,\tau)=\sum_{a\in \mathfrak{a}}e^{2\pi i(N(\mathfrak{a})^{-1}N(a)\tau+aw)},
\]
for $w\in \Cc$ and $\tau\in \h$. This defines a holomorphic function $$\vartheta_\mathfrak{a}:\Cc\times \h\to \Cc.$$
\begin{Proposition}
The function $\vartheta_{\mathfrak{a}}$ satisfies the functional equations:
\begin{enumerate}[(i)]
\item For $\lambda\in K^\times$, we have $\vartheta_{\lambda\mathfrak{a}}(w,\tau)=\vartheta_{\mathfrak{a}}(\lambda w,\tau)$.
\item For $\gamma\in \Gamma_0(D)$, we have
\[
\vartheta_{\mathfrak{a}}(j(\gamma,\tau)^{-1}w,\gamma\tau)=\omega_{K/\Q}(\gamma)j(\gamma,\tau)\vartheta_{\mathfrak{a}}(w,\tau).
\]
\item For $\tau_0\in \h\cap K$, the function $\vartheta_{\mathfrak{a}}(w,\tau_0):\Cc\to \Cc$, is a theta function, in the sense of Shimura, with respect to a certain lattice $L_{\tau,\mathfrak{a}}\subset K$.
\end{enumerate}
\end{Proposition}

\begin{proof}
See \cite[Proposition 2.4]{finis}.
\end{proof}

\noindent Finis's result on the Fourier-Jacobi expansion of the lifted form can be expressed in terms of these intrinsic theta functions:
\begin{Theorem}[Finis]
Let $f\in S_{k-1}(\Gamma_1(D),\omega_{K/\Q})$. Then the lifted form $\mathcal{L}_{k,\varepsilon,\chi}(f)$ has Fourier-Jacobi coefficients
\[
g_{n/N(\mathfrak{a}),\mathfrak{a}}(w)=\varepsilon(\overline{\mathfrak{a}}/\mathfrak{a})\sum_{\mathfrak{b}\in \emph{Cl}_{K}}(\varepsilon^3\chi^{-1})(\mathfrak{b})N(\mathfrak{b})^{k/2}T_n(\emph{Tr}(f(\tau)\vartheta_{\overline{\mathfrak{b}}}(n\delta w,\tau)))|_{\tau=\tau_{\mathfrak{a},\mathfrak{b}}(\mathfrak{a}\overline{\mathfrak{b}})}.
\]
Here $\mathfrak{b}$ ranges over a system of representatives for the ideal classes of $K$ such that $\mathfrak{a}\overline{\mathfrak{b}}=\Z+\Z\tau_{\mathfrak{a},\mathfrak{b}}$ for some $\tau_{\mathfrak{a},\mathfrak{b}}\in K\cap\h$ and $\emph{Tr}$ denotes $\emph{Tr}_{\Gamma_1(D)\backslash\SL_2(\Z)}$.
\end{Theorem}

\begin{proof}
See \cite[Theorem 2.5]{finis}.
\end{proof}

\begin{Theorem}
Let $k,\varepsilon$ and $\chi$ be as before. There exists a period $\Omega_0\in \Cc$ only depending on $K$, such that, $\mathcal{L}^{\emph{ar}}_{k,\varepsilon,\chi}=\Omega_0^{-k}\mathcal{L}_{k,\varepsilon,\chi}$ is arithmetic in the sense of Shimura, i.e. we have a morphism
\[
\mathcal{L}^{\emph{ar}}:S_{k-1}(\Gamma_1(D),\omega_{K/\Q},\overline{\Z}[1/D])\to \mathscr{S}^{\emph{ar}}(k,L,\chi).
\]
\end{Theorem}

\begin{proof}
See \cite[Theorem 3.5]{finis}.
\end{proof}

The following definition of Hecke operators for $G=\GU(2,1)$ is recalled in \cite[p. 325]{kudla1}; in the automorphic context, this is also recalled in \cite[p. 8]{murase}. As Kudla in \textit{loc. cit.} works in the class number one context, we follow the more general notations of Finis's thesis. For each prime ideal $\mathfrak{q}$ of $\mathcal{O}_{K}$, lying over the prime $q$ of $\Q$, let us define a (non-normalized) Hecke operator $T_{\mathfrak{q}}^{\text{naive}}$ acting on $\mathscr{S}(k,L,\chi)$. For $F\in \mathscr{S}(k,L,\chi)$,
\[
T_{\mathfrak{q}}^{\text{naive}}F(g_0)=\int_{S_{\mathfrak{q}}}F(gg_0)dg,
\] 
where $S_{\mathfrak{q}}$ is a subset of $G(\Q_q)$ defined as follows: if $\mathfrak{q}= \overline{\mathfrak{q}}$,
\[
S_{\mathfrak{q}}=\{g\in G(\Q_q)\mid g(L\otimes \Z_q)\subseteq (L\otimes \Z_q), v_q(\mu(g))=v_q(\text{N}(\mathfrak{q}))\},
\]
where $v_q$ is the valuation associated to $q$, and if $q$ splits in $K$,
\[
S_{\mathfrak{q}}=\{g\in G(\Q_q)\mid g(L\otimes \Z_q)\subseteq (L\otimes \Z_q), \det g\in \mathfrak{q}\overline{\mathfrak{q}}^2\otimes\Z_q\}.
\]
If $F$ is a Hecke eigenform with central character $\chi$ and eigenvalues $\tilde{\lambda}(\mathfrak{q}):=\lambda^{\text{naive}}(\mathfrak{q})$, Shintani defined an $L$-function $\zeta(F,\xi,s)$ for $\xi$ a Hecke character of $K$: $\zeta(F,\xi,s)=\prod_{\mathfrak{q}}P_{\mathfrak{q}}(\xi(\mathfrak{q})\text{N}(\mathfrak{q})^{-s-1})^{-1}$ with (here $N=\text{N}(\mathfrak{q})$)
\[
P_{\mathfrak{q}}(x):=\begin{cases}
1-\tilde{\lambda}(\overline{\mathfrak{q}})\chi(\mathfrak{q})^{-1}N^{-1}x+\tilde{\lambda}(\mathfrak{q})\chi(\mathfrak{q})^{-1}N^{-1}x^2-\chi(\overline{\mathfrak{q}}\mathfrak{q}^{-1})x^3, \! & \mathfrak{q}\ne \overline{\mathfrak{q}}, \\
(1-x)(1+(N^{1/2}-\tilde{\lambda}(\mathfrak{q})\chi(\mathfrak{q})^{-1})N^{-1}x+x^2), & \mathfrak{q}= q\mathcal{O}_{K},\\
(1-x)(1+(N-\tilde{\lambda}(\mathfrak{q})\chi(\mathfrak{q})^{-1})N^{-1}x+x^2), & \mathfrak{q}\mid \delta \mathcal{O}_{K}.
\end{cases}
\]

\begin{Proposition}
\label{eigen}
Suppose that $f\in S_{k-1}(\Gamma_1(D),\omega_{K/\Q})$ is a Hecke eigenform with eigenvalues ${a_q}$, whose lifting $\mathcal{L}_{k,\varepsilon,\chi}(f)$ is non-zero. Then the lift is also a Hecke eigenform with eigenvalues
\[
\lambda^{\emph{naive}}(\mathfrak{q},f)=\lambda^{\emph{naive}}(\mathfrak{q})=\begin{cases}
q^{2-k/2}\varepsilon(\overline{\mathfrak{q}})a_q+q(\chi\varepsilon^{-2})(\overline{\mathfrak{q}}), & \mathfrak{q}\ne \overline{\mathfrak{q}}, \\
q^{2-k/2}\varepsilon(\overline{\mathfrak{q}})(a_q+\overline{a}_q)+q(\chi\varepsilon^{-2})(\overline{\mathfrak{q}}), & \mathfrak{q}\mid \delta \mathcal{O}_{K},\\
q^{4-k}a_q^2+2q^{2}+q, & \mathfrak{q}= q\mathcal{O}_{K},
\end{cases}
\]
and we have the relation
\[
\zeta(\mathcal{L}(f),\xi,s)=L(f_{K}\otimes\varepsilon\chi^{-1},\xi,s+k/2)L(\varepsilon^{-2}\xi,s+1),
\]
where $f_{K}$ denotes the base change of $f$ to $\GL(2,K)$, the $L$-function appearing in the first factor in the right hand side is the Dirichlet series attached to the form $f_K$ and $L(\xi, s)$ is the $L$-series for the Hecke character $\xi$.
\end{Proposition}

\begin{proof}
See \cite[Theorem 1.1]{kudla1}.
\end{proof}

\begin{Corollary}
\begin{enumerate}[(i)]
\item The normalized operator $T_{\mathfrak{q}}:=\emph{N}(\mathfrak{q})^{k/2-1}T_{\mathfrak{q}}^{\emph{naive}}$ is integral.
\item The eigenvalues for the normalized operator $T_{\mathfrak{q}}$ are:
\[
\lambda(\mathfrak{q})=\lambda^{\emph{naive}}(\mathfrak{q})\emph{N}(\mathfrak{q})^{k/2-1}=\begin{cases}
\varepsilon(\overline{\mathfrak{q}})q a_q+q^{k/2}(\chi\varepsilon^{-2})(\overline{\mathfrak{q}}), & \mathfrak{q}\ne \overline{\mathfrak{q}}, \\
\varepsilon(\overline{\mathfrak{q}})q(a_q+\overline{a}_q)+q^{k/2}(\chi\varepsilon^{-2})(\overline{\mathfrak{q}}), & \mathfrak{q}\mid \delta \mathcal{O}_{K},\\
q^2 a_q^2+2q^{k}+q^{k-1}, & \mathfrak{q}= q\mathcal{O}_{K}.
\end{cases}
\]
\end{enumerate}
\end{Corollary}
\begin{proof}
The normalization introduced agrees with that in \cite[p. 165]{finis2}, where it is proved that $T_{\mathfrak{q}}$ is optimally integral, meaning that there exist eigenforms in $\mathscr{S}(k,L)$ whose $T_{\mathfrak{q}}$-eigenvalues are not divisible by $q$, see \cite[p. 178]{finis2}.
\end{proof}
\begin{remark}
The normalization in \cite{finis2} agrees with that of \cite{hsieh}, as both are optimally integal. Note that the lifted forms are never $T_{\mathfrak{q}}$-ordinary. If 
\end{remark}

\subsection{The $p$-modification}
In this subsection, for $r\geq 1$ a fixed integer, we take a lattice $L:=p^rD\mathcal{O}_{K}\oplus\mathcal{O}_{K}\oplus p^rD\mathcal{O}_{K}$ with dual $L^*:=(p^{r}D)^{-1}\mathcal{O}_{K}\oplus \frac{1}{2}\mathcal{O}_{K}\oplus(p^{r}D)^{-1}\mathcal{O}_{K}$. We also fix a Nebentypus character $\xi_r:(\Z/p^r\Z)^\times\to \Cc^\times$.
We construct a theta lift:
\[
\mathcal{L}^p:S_{k-1}(\Gamma_1(p^rD),\omega_{K/\Q}\xi_r)\to\mathscr{S}(k, \Gamma, \chi),
\]
for a certain subgroup $\Gamma\subset \Gamma(L)$ and a Hecke character $\chi$ with infinity type $(k/2,-k/2)$, unramified away from $p$, and such that the finite part of $\chi$ restricted to $\Z$ coincides with $\xi_r^{-1}$. For the definition of the theta kernel, we follow Kudla \cite[p. 12]{kudlaq}, but we adapt instead the notation and formulation of Finis. First, we recall this standard result:
\begin{Lemma}
Let $A$ be a Dedekind domain and let $\mathfrak{a}$ be an $A$-ideal. Every ideal class in
$\emph{Cl}_A$ can be represented by an ideal coprime to $\mathfrak{a}$.
\end{Lemma}

\begin{proof}
Let $I$ be a nonzero fractional ideal of $A$. For each prime $\mathfrak{p}\mid \mathfrak{a}$ we can pick $\pi_{\mathfrak{p}}\in \mathfrak{p}$ such
that $v_{\mathfrak{q}}(\pi_{\mathfrak{p}})=v_{\mathfrak{q}}(\mathfrak{p})$ for all $\mathfrak{q}\mid \mathfrak{a}$. If we then put $\alpha:=\prod_{\mathfrak{p}\mid \mathfrak{a}}\pi_{\mathfrak{p}}^{-v_{\mathfrak{p}}(I)}$ , then $v_{\mathfrak{p}}(\alpha I)=0$ for all $\mathfrak{p}\mid \mathfrak{a}$; thus $\alpha I$ is coprime to $\mathfrak{a}$ and $[\alpha I] = [I]$.
\end{proof}
For each $\mathfrak{b}\in \text{Cl}_K$ we may thus pick a representative in the class which is coprime to $p$. For any $h\in L^*/L$, we define the theta kernel $\theta_{h,\mathfrak{b}}=\theta_{h, \mathfrak{b}, k,\varepsilon,\chi}$ on $\h\times G(\A_\Q)$ by
\[
\overline{\theta_{h,\mathfrak{b}}(\tau, g)}=2^{k}D^{-k/2}y^2\eta(\mathfrak{z}_0)^{-k}\mu(g_\infty)^{-k/2}\sum_{\substack{X\in \mathfrak{b} g_fL^*\\X\equiv h\, (L)}}\overline{f_{\tau,\mathfrak{z}}^{\mathfrak{b}g_f}(X)},
\]
where
\[
f_{\tau,\mathfrak{z}}^{\mathfrak{b}g_f}(X)=(X,g_\infty \underline{\mathfrak{z}_0})^ke^{2\pi i|N(\mathfrak{b})\mu(g_f)|_{\A_{\Q}} (2iy|(X,g_\infty \underline{\mathfrak{z}_0})|^2(\mu(g_\infty)\eta(\mathfrak{z}_0))^{-1}+\overline{\tau}(X,X))},
\]
so that
\[
\overline{\theta_h(\tau, g)}=\sum_{\substack{\mathfrak{b}\in \text{Cl}_{K}\\ (\mathfrak{b},p)=1}}\varepsilon(\det g)|N(\mathfrak{b})\mu(g_f)|_{\A_{\Q}}^{k/2}(\varepsilon^3\chi^{-1})(\mathfrak{b})\overline{\theta_{h,\mathfrak{b}}(\tau, g)},
\]
where the notation for the summation means that we have chosen, for each ideal class, a representative coprime to $p$. Clearly, the choice of the fractional ideal is not unique, but the theta kernel $\theta_h$ does not depend on such choice, by the same argument as in \cite[p. 9]{finis}. To obtain our formula, we relate our lifting to the lifting from $\GU(2)$ to $\GU(2)$ introduced by Kudla \cite[p. 14]{kudlaq}, a technique already used by Finis. To this end, consider the two-dimensional hermitian space $V_0=K^2$ with hermitian form of matrix
\[
\begin{pmatrix}
0&\delta^{-1}\\
-\delta^{-1}&0
\end{pmatrix}
\]
Then the associated special unitary group is $\SU(V_0)\equiv\SL_2(\R)$. Fix $L_0=(p^rD\mathcal{O}_{K})^2$ and $L_0^*=((p^rD)^{-1}\mathcal{O}_{K})^2$, and set
\[
\tilde{\Gamma}=\Gamma(p^{2r}D^2)=\{\gamma\in \SL_2(\Z)\mid \gamma L_0=L_0 \;\text{and}\; \gamma\; \text{is trivial on } L_0^*/L_0\}.
\]
For a fractional ideal $\mathfrak{c}$ of $K$ and $h\in L_0^*/L_0$ a theta kernel $\h\times \h\to \Cc$ is given by 
\[
\Theta^0_{k,\mathfrak{c},h}(\tau,z):=2^{k}D^{-k/2}y\sum_{\substack{X\in \overline{\mathfrak{c}} L_0^*\\\overline{\mathfrak{c}}^{-1}X\equiv h\, (L_0)}}\Phi_{k,\mathfrak{c}}(X,\tau,z).
\]
where
\[
\Phi_{k,\mathfrak{c}}(X,\tau,z):=(X,(z,1))^k\eta^0(z)^{-k}\\
e^{2\pi i N(\mathfrak{c})^{-1} (2iy|(X,(z,1))|^2\eta^0(z)^{-1}+\overline{\tau}(X,X))},
\]
with $\eta^0(z)=(2/\sqrt{D})\text{Im}(z)$.
For $h=(h_1,h_2)\in L_0^*$, call $\gamma_h\in M_2(\Z)$ the unique matrix such that
\[
p^rD(h_1,h_2)=(1,\tau_K)\gamma_h.
\]
For $n\in \Z$, let 
\[
\mathcal{H}_h^{=n}=\{\gamma\in M_2(\Z)\mid \det \gamma=n, \gamma\equiv \gamma_h \!\mod p^{2r}D^2\}.
\]
For each $\gamma=\begin{pmatrix}
a&b\\
c&d
\end{pmatrix}\in M_2(\Z)$ such that $\det \gamma<0$, write $\gamma''=\begin{pmatrix}
d&b\\
-c&-a
\end{pmatrix}\in M_2(\Z)$. The following result is essentially a restatement of a result of Kudla, adapted to our setting:

\begin{Proposition}
\label{constant}
Let $\mathfrak{c}=\Z+\Z\tau_{\mathfrak{c}}$, $\tau_{\mathfrak{c}}\in \h$, be a fractional ideal of $\mathcal{O}_K$, $k\geq 4$ be an even integer and $h\in L_0^*/L_0$ such that $h_1=0$. For $f\in S_k(\Gamma(p^{2r}D^2))$, let $\tilde{\mathcal{L}}^0_{k,\mathfrak{c},h}(f)=\langle f,\Theta^0_{k,\mathfrak{c},h}\rangle \in S_k(\Gamma(p^{2r}D^2))$. Writing $\tilde{\mathcal{L}}^0_h(f)=\sum_{n=1}^{\infty} a_nq^n$, for $q=e^{2\pi inz/p^{2r}D^2}$, we have
\[
a_n=\emph{N}(\mathfrak{c})^k\sum_{t\in p^{-r}D^{-1}\mathcal{O}_K (p^rD\mathcal{O}_K)}\sum_{\gamma\in \mathcal{H}_{(h_2,t)}^{=-n}/ \Gamma(p^{2r}D^2)}(f\mid_k \gamma'')\mid_{\tau=-\overline{\tau}_{\mathfrak{c}}}.
\]
\end{Proposition}

\begin{proof}
This is the application of Kudla's result \cite[Theorem 6.6]{kudlaq} for $M=p^rD$. Note that since we are assuming $h_1=0$, the exponentials in Kudla's formula reduce to $1$ for each $t$.
\end{proof}

\noindent Now define
\[
\mathcal{H}=\biggl\{\begin{pmatrix}
a&b\\
c&d
\end{pmatrix}\in M_2(\Z)\mid c\equiv 0 \!\!\mod p^{2r}D^2, a\equiv -1\!\!\mod p^{2r}D^2 \biggr\}
\]
and for each $n\in \Z$,
\[
\mathcal{H}^{=n}:=\mathcal{H}\cap \{\gamma\in M_2(\Z)\mid \det \gamma=n\}.
\]
Note that we have a disjoint union:
\[
\mathcal{H}^{=n}=\dot\bigcup_{t\in p^{-r}D^{-1}\mathcal{O}_K/p^{r}D\mathcal{O}_K} \mathcal{H}^{=n}_{(-p^{-r}D^{-1},t)}.
\]
For $f\in S_k(\Gamma(p^{2r}D^2))$, define 
\[
\mathcal{L}^0(f):=[\Gamma_1(p^{2r}D^2):\Gamma(p^{2r}D^2)]^{-1}\tilde{\mathcal{L}}^0_{(0,-p^{-r}D^{-1})}(f).
\]
\begin{Theorem}
Let $\mathfrak{c}=\Z+\Z\tau_{\mathfrak{c}}$, $\tau_{\mathfrak{c}}\in \h$, be a fractional ideal of $\mathcal{O}_K$ and $k\geq 4$ be an even integer. For $f\in S_k(\Gamma(p^{2r}D^2))$, write $\mathcal{L}^0(f)=\sum_{n=1}^{\infty} a_nq^n\in S_k(\Gamma(p^{2r}D^2))$. We have
\[
a_n=[\Gamma_1(p^{2r}D^2):\Gamma(p^{2r}D^2)]^{-1}\emph{N}(\mathfrak{c})^kT_n^*(\emph{Tr}_{\Gamma(p^{2r}D^2)\backslash \Gamma_1(p^{2r}D^2)}f)\mid_{\tau=-\overline{\tau}_{\mathfrak{c}}},
\]
where $T_n^*$ is the adjoint of the Hecke operator at $n$ of level $\Gamma_1(p^{2r}D^2)$.
\end{Theorem}

\begin{proof}
First of all, notice that, for $n\in \N$,
\[
\mathcal{H}^{=-n}/\Gamma(p^{2r}D^2)=\dot\bigcup_{t\in p^{-r}D^{-1}\mathcal{O}_K/p^rD\mathcal{O}_K}\mathcal{H}_{(-p^{-r}D^{-1},t)}^{=-n}/\Gamma(p^{2r}D^2).
\]
Consider first the case where $n=q^s$ is the power of a prime $q$. If $\gamma_1, \dots ,\gamma_l$ is a set of representatives for $\mathcal{H}^{=-n}/\Gamma(p^{2r}D^2)$, then $\gamma_1'',\dots ,\gamma_l''$ is a set of representatives for the left quotient by $\Gamma(p^{2r}D^2)$ of the set
\[
\biggl\{\begin{pmatrix}
a&b\\
c&d
\end{pmatrix}\in M_2(\Z)\mid ad-bc=n, c\equiv 0 \!\!\mod p^{2r}D^2, d\equiv 1\!\!\mod p^{2r}D^2 \biggr\}.
\]
We have two possibilities:
\begin{enumerate}[(i)]
\item Suppose $q \nmid p^{2r}D^2$. Then the above set is clearly equal to the double coset $\Gamma_1(p^{2r}D^2)\begin{pmatrix}
n&\\
&1
\end{pmatrix}\Gamma_1(p^{2r}D^2)$.
\item If $q\mid p^{2r}D^2$, then for each matrix $\begin{pmatrix}
a&b\\
c&d
\end{pmatrix}$ in the above set we also see that $a\equiv 0\!\! \mod \text{gcd}(n,p^{2r}D^2)$, so that again the above is just the double coset $\Gamma_1(p^{2r}D^2)\begin{pmatrix}
n&\\
&1
\end{pmatrix}\Gamma_1(p^{2r}D^2)$.
\end{enumerate}
In both cases, from a set of representatives of $\mathcal{H}^{=-n}/\Gamma(p^{2r}D^2)$, we get a set of representatives for
\[
\Gamma(p^{2r}D^2)\backslash \Gamma_1(p^{2r}D^2)\begin{pmatrix}
n&\\
&1
\end{pmatrix}\Gamma_1(p^{2r}D^2).
\]
Hence, if $g$ is a modular form of level $\Gamma(p^{2r}D^2)$ and weight $k$,
\[
\sum_{\gamma\in \mathcal{H}^{=-n}/\Gamma(p^{2r}D^2)} g\mid_{k}\gamma''=T_n^*(\text{Tr}_{\Gamma(p^{2r}D^2)\backslash \Gamma_1(p^{2r}D^2)}g).
\]
Then, we have for the $n$-th coefficient of $\mathcal{L}^0(f)$,
\begin{align*}
a_n=[\Gamma_1(p^{2r}D^2):&\Gamma(p^{2r}D^2)]^{-1}\text{N}(\mathfrak{c})^k\\
&\sum_{t\in p^{-r}D^{-1}\mathcal{O}_K/p^rD\mathcal{O}_K}\sum_{\gamma\in \mathcal{H}_{(-p^{-r}D^{-1},t)}^{=-n}/\Gamma(p^{2r}D^2)}(f\mid_k\gamma'')\mid_{\tau=-\overline{\tau}_{\mathfrak{c}}}.
\end{align*}
But now
\[
\sum_{t\in p^{-r}D^{-1}\mathcal{O}_K/p^rD\mathcal{O}_K}\sum_{\gamma\in \mathcal{H}_{(-p^{-r}D^{-1},t)}^{=-n}/\Gamma(p^{2r}D^2)}(f\mid_k\gamma'')\mid_{\tau=-\overline{\tau}_{\mathfrak{c}}}=\sum_{\gamma\in \mathcal{H}^{=-n}/\Gamma(p^{2r}D^2)} f\mid_{k}\gamma'',
\]
and applying the above argument we get the formula for $n$ power of a prime. The general case follows from the above by using the multiplicativity of the double cosets.
\end{proof}
Recall that, for $N\in \N^*$, we have a congruence subgroup of $\SL_2(\Z)$:
\[
\Gamma^1(N)=\Bigg\{\begin{pmatrix}
a&b\\
c&d
\end{pmatrix}\in \SL_2(\Z)\,\,\Bigg|\,\, a,d\equiv 1\!\!\mod\, N,\, b\equiv 0\!\!\mod\, N\Bigg\}.
\]
\begin{Proposition}
\label{-1}
The lifted form $\mathcal{L}^0(f)$ belongs to $S_k(\Gamma^1(p^{2r}D^2))$.
\end{Proposition}

\begin{proof}
For $\gamma\in \Gamma^1(p^{2r}D^2)$, the computation in \cite[p. 329, (2.2)]{kudla1} tells us that 
\[
\Theta^0_{k,\mathfrak{c},h}\mid_k \gamma=\Theta^0_{k,\mathfrak{c},\gamma^{-1}h},
\]
but $\gamma^{-1}\in \Gamma^1(p^{2r}D^2)$ fixes $(0,-p^{-r}D^{-1})\in L_0^*/L_0$.
\end{proof}

\begin{Remark}
As
\[
\begin{pmatrix}
1/p^{2r}D^2&0\\
0&1
\end{pmatrix}\Gamma^1(p^{2r}D^2)\begin{pmatrix}
p^{2r}D^2&0\\
0&1
\end{pmatrix}=\Gamma_1(p^{2r}D^2),
\]
we get a Hecke equivariant isomorphism $S_k(\Gamma^1(p^{2r}D^2))\cong S_k(\Gamma_1(p^{2r}D^2))$ by $f(z)\mapsto f(p^{2r}D^2z)$. Thus we can think of our lifts as lying in the more classical space $S_k(\Gamma_1(p^{2r}D^2))$. 
\end{Remark}
The above immediately implies:
\begin{Corollary}
The lift $\mathcal{L}^0$ restricts to an algebraic endomorphism of $S_k(\Gamma_1(p^{2r}D^2))$, and the formula for the coefficients of $\mathcal{L}^0(f)\in S_k(\Gamma_1(p^{2r}D^2))$ is
\[
a_n=\emph{N}(\mathfrak{c})^kT_n^*f\mid_{\tau=-\overline{\tau}_{\mathfrak{c}}}.
\]
\end{Corollary}
\noindent Additionally, as $\Gamma_1(p^{2r}D^2)\subset \Gamma_1(p^{r}D)$, we may further restrict our lift to the space $S_k(\Gamma_1(p^{r}D))$.
For $f\in S_{k-1}(\Gamma_1(p^{2r}D^2),\omega_{K/\Q}\xi_{2r})$ define, for $h_0:=(0,0,-p^{-r}D^{-1})$,
\[
\mathcal{L}^p(f)=[\Gamma_1(p^{2r}D^2):\Gamma(p^{2r}D^2)]^{-1}\langle f, \theta_{h_0}\rangle.
\]
\noindent Consider the group of matrices in $G(L)_f\cap G(\Q)$ which are in the lower unipotent modulo $p^{2r}D^2$.
We will denote this group by $\Gamma_0^1(p^{2r}D^2)$.
\begin{Proposition}
For $f\in S_{k-1}(\Gamma_1(p^{2r}D^2),\omega_{K/\Q}\xi_{2r})$, we have $$\mathcal{L}^p(f)\in \mathscr{S}(k,\Gamma_0^1(p^{2r}D^2),\chi).$$ 
\end{Proposition}

\begin{proof}
The congruence subgroup $\Gamma_0^1(p^{2r}D^2)$ stabilizes $h_0$, and thus, by the proof of Proposition \ref{-1}, we see that $\mathcal{L}^p(f)$ is invariant under the slash action of $\Gamma_0^1(p^{2r}D^2)$. The fact that $\mathcal{L}^p(f)$ has weight $k$ and central character $\chi$ follows by \cite[Theorem 5.3]{kudlaq} and \cite[Proposition 2.1]{finis}.
\end{proof}

\begin{Remark}
The stabilizer of $h_0$ in $G(\Q)$ is actually bigger than $\Gamma_0^1(p^{2r}D^2)$, but for our purposes this will suffice.
\end{Remark}

\begin{Theorem}
Let $k\geq 6$ and $f\in S_{k-1}(\Gamma_1(p^{2r}D^2),\omega_{K/\Q}\xi_{2r})$. Then the lifted form $\mathcal{L}^p_{k,\varepsilon,\chi}(f)$ has Fourier-Jacobi coefficients
\[
g_{n/N(\mathfrak{a}),\mathfrak{a}}(w)=\varepsilon(\overline{\mathfrak{a}}/\mathfrak{a})\sum_{\substack{\mathfrak{b}\in \emph{Cl}_{K}\\ (\mathfrak{b},p)=1}}(\varepsilon^3\chi^{-1})(\mathfrak{b})N(\mathfrak{b})^{k/2}T_n^*(f(\tau)\vartheta_{\overline{\mathfrak{b}}}(n\delta w,\tau))|_{\tau=-\overline{\tau}_{\mathfrak{a},\mathfrak{b}}}.
\]
Here $\mathfrak{b}$ ranges over a system of representatives for the ideal classes of $K$ such that $\mathfrak{a}\overline{\mathfrak{b}}=\Z+\Z\tau_{\mathfrak{a},\mathfrak{b}}$ for some $\tau_{\mathfrak{a},\mathfrak{b}}\in K\cap\h$.
\end{Theorem}

\begin{proof}
We adapt Finis's argument to our case. To compute the Fourier-Jacobi coefficients $g_{n/\text{N}(\mathfrak{a}),\mathfrak{a}}$ of $F=\mathcal{L}^p(f)$ set $$g_f=\begin{pmatrix}
\overline{a}_f&&\\&1&\\&&a_f^{-1}\end{pmatrix}\in G(\A_{\Q,f})$$ for $a_f\in \mathbb{A}_{K,f}^\times$ with corresponding fractional ideal $\mathfrak{a}$, and consider the Fourier-Jacobi expansion of $F_{g_f}$:
\[
F_{g_f}(\mathfrak{z})=\sum_{n\in \Q}g_{n,\mathfrak{a}}(w)q^n.
\]
We have
\[
F_{g_f}(\mathfrak{z})=\varepsilon(\overline{\mathfrak{a}}/\mathfrak{a})\sum_{\substack{\mathfrak{b}\in \text{Cl}_{K}\\ (\mathfrak{b},p)=1}}(\varepsilon^3\chi^{-1})(\mathfrak{b})N(\mathfrak{b})^{-k/2}F_{\mathfrak{a},\mathfrak{b}}(\mathfrak{z}),
\]
where $F_{\mathfrak{a},\mathfrak{b}}(f)(\mathfrak{z})$ is the lifting of $f$ against the theta kernel $\Theta_{\mathfrak{a},\mathfrak{b}}$ defined as
\begin{align*}
\Theta_{\mathfrak{a},\mathfrak{b}}(\tau, \mathfrak{z})=&2^{k}D^{-k/2}y^2c_r\sum_{\substack{X\in (p^rD)^{-1}\overline{\mathfrak{a}}\mathfrak{b}\oplus \mathfrak{b}\oplus (p^rD)^{-1}\mathfrak{a}^{-1}\mathfrak{b}\\(g_f\mathfrak{b})^{-1}X\equiv h_0\, (L)}}(X,\underline{\mathfrak{z}})^k\eta(\mathfrak{z})^{-k}\\
& e^{2\pi iN(\mathfrak{b})^{-1}|\mu(g_f)|_{\A_{\Q}} (2iy\eta(\mathfrak{z})^{-1}|(X,\underline{\mathfrak{z}})|^2+\overline{\tau}(X,X))},
\end{align*}
where $h_0=(0,0,-p^{-r}D^{-1})\in L^*/L$ and $c_r:=[\Gamma_1(p^{2r}D^2):\Gamma(p^{2r}D^2)]^{-1}$.
By considering Taylor coefficients at zero, we are reduced to computing the Fourier expansions of the scalar valued functions
\[
\phi_{\mathfrak{a},\mathfrak{b}}^{(m)}(z):=\frac{1}{m!}\frac{\partial^m}{\partial w^m}F_{\mathfrak{a},\mathfrak{b}}(z,w)\bigg|_{w=0}.
\]
In a first step, we express these functions in terms of the so-called two-dimensional lifting $\mathcal{L}^0$. Evidently, $\phi^{(m)}_{\mathfrak{a},\mathfrak{b}}=\langle  f,\Theta_{\mathfrak{a},\mathfrak{b}}^{(m)}\rangle$ with
\[
\overline{\Theta^{(m)}_{\mathfrak{a},\mathfrak{b}}(\tau,z)}=\frac{1}{m!}\frac{\partial^m}{\partial w^m}\overline{\Theta_{\mathfrak{a},\mathfrak{b}}\big(\tau,(z,w)\big)}\bigg|_{w=0}.
\]
Explicitly, we have
\begin{align*}
\overline{\Theta_{\mathfrak{a},\mathfrak{b}}(\tau,\mathfrak{z})}=&2^{k}D^{-k/2}y^2c_r\sum_{\substack{X\in (p^rD)^{-1}\overline{\mathfrak{a}}\mathfrak{b}\oplus \mathfrak{b}\oplus (p^rD)^{-1}\mathfrak{a}^{-1}\mathfrak{b}\\(g_f\mathfrak{b})^{-1}X\equiv h_0\, (L)}}(\overline{(X^0,(z,1))}-X_2\overline{w})^k\\
 (\eta^0(z)-&|w|^2)^{-k}e^{-2\pi iN(\mathfrak{b})^{-1}(-2iy|(X^0,(z,1))-\overline{X}_2w|^2(\eta^0(z)-|w|^2)^{-1}+\tau((X^0,X^0)-\text{N}(X_2)))},
\end{align*}
with $X^0:=(X_1,X_3)$. Taking holomorphic derivatives with respect to $w$ yields
\begin{align*}
\overline{\Theta^{(m)}_{\mathfrak{a},\mathfrak{b}}(\tau,z)}=&2^{k}D^{-k/2}y^{2+m}c_r\text{N}(\mathfrak{a})^{-k-m}\frac{1}{m!}\bigg(\frac{4\pi}{\text{N}(\mathfrak{b})}\bigg)^m\\
&\Bigg(\sum_{\substack{Y\in \overline{\mathfrak{c}}L_0^*\\\overline{\mathfrak{c}}^{-1}Y\equiv h_0\,(L_0)}}\Phi_{k+m,\mathfrak{c}}(Y,\tau,z')\Bigg)\Bigg(\sum_{X_2\in \mathfrak{b}}\overline{X}_2^{m}e^{2\pi i\text{N}(b)^{-1}\text{N}(X_2)\tau}\Bigg),
\end{align*}
where we set $Y=(X_1,\text{N}(\mathfrak{a})X_3), z'=\text{N}(\mathfrak{a})^{-1}z$ and $\mathfrak{c}=\mathfrak{a}\overline{\mathfrak{b}}$. Now notice how the first summation just corresponds to the conjugation of the theta kernel $\Theta^0_{k+m,\mathfrak{c}}(\tau,z')$ and the second summation can be rewritten as $\frac{\partial^{m}}{\partial w^m}\vartheta_{\overline{\mathfrak{b}}}$, for $\vartheta_{\overline{\mathfrak{b}}}$ the theta function defined in the previous section, \ref{theta}. By results of Hecke, see \cite[p. 237]{petersson}, $\vartheta^{(m)}_{\overline{\mathfrak{b}}}$ is a modular form of weight $m+1$ and character $\omega_{K/\Q}$ for the group $\Gamma_0(D)$.
After some computations which are carried exactly in the same way in \cite[p. 19-20]{finis}, we can now rewrite
\begin{align*}
\phi^{(m)}_{\mathfrak{a},\mathfrak{b}}(z)=\text{N}(\mathfrak{a})^{-k-m}\text{N}(\mathfrak{b})^{-m}\frac{(2\pi \sqrt{D})^m}{m!}\mathcal{L}^0_{k+m,\mathfrak{c}}\big(f\vartheta^{(m)}_{\overline{\mathfrak{b}}}\big)(z').
\end{align*}
Summing over all representatives $\mathfrak{b}$, we have computed the $m$-th Taylor coefficient of $g_{n/\text{N}(\mathfrak{a}),\mathfrak{a}}$ at $w=0$. The result easily follows by putting together the Taylor expansion $F_{\mathfrak{a},\mathfrak{b}}$ of $\vartheta_{\overline{\mathfrak{b}}}$:
\[
\sum_{m\geq 0}\phi_{\mathfrak{a},\mathfrak{b}}^{(m)}w^m=\sum_{m\geq 0}\text{N}(\mathfrak{a})^{-k-m}\text{N}(\mathfrak{b})^{-m}\frac{(2\pi \sqrt{D})^m}{m!}\mathcal{L}^0_{k+m,\mathfrak{c}}\big(f\vartheta^{(m)}_{\overline{\mathfrak{b}}}w^m\big)(z').
\]
\end{proof}

\begin{Corollary}
Let $k,\varepsilon$ and $\chi$ be as before. There exists a complex period $\Omega_0$ such that $\mathcal{L}^{p,\emph{ar}}_{k,\varepsilon,\chi}=\Omega_0^{-k}\mathcal{L}^p_{k,\varepsilon,\chi}$ is arithmetic, i.e. 
$$f\in S_{k-1}(\Gamma_1(p^{2r}D^2),\omega_{K/\Q}\xi_{2r},\overline{\Z}[1/D])\Rightarrow\mathcal{L}^{p,\emph{ar}}(f)\in \mathscr{S}(k,\Gamma_0^1(p^{2r}D^2),\chi, \overline{\Z}[1/D]).$$
\end{Corollary}

\begin{proof}
Finis's arguments, see \cite[p. 36]{finis}, still apply to our simpler formulas, as the theta functions $\vartheta$ are the same in both cases.
\end{proof}

We now show that, much as in the level $1$ case, the $p$-modified lift is Hecke equivariant, and in fact that a completely analogous statement to Proposition \ref{eigen} can be obtained. For this we will adapt Kudla's proof given in \cite[\S 2]{kudla1}, which draws from Eichler's classical technique \cite[2, Satz 21.3]{eichler} to relate the actions of Hecke operators and generalized Brandt matrices on spaces of theta functions.

\begin{Theorem}
\label{peigen}
If $f\in S_{k-1}(\Gamma_1(p^{2r}D^2),\omega_{K/\Q}\xi_{2r})$ is a Hecke eigenform with eigenvalues $a_p$ and the lifting $\mathcal{L}^p_{k,\varepsilon, \chi}(f)$ is non-zero, then the latter is again a Hecke eigenform with eigenvalues:
\[
\lambda^p(\mathfrak{q})=\begin{cases}
p\varepsilon(\overline{\mathfrak{p}})a_p, & \mathfrak{q}=\mathfrak{p},\\
\lambda(\mathfrak{q}), & \text{if}\;\, q\ne p\;\, \text{is unramified in } K,
\end{cases}
\]
where $\lambda(\mathfrak{q})$ are the values defined in section \ref{eigen}.
\end{Theorem}

\begin{proof}
For $g\in G(\Q)$, let $g^\iota=\mu(g)g^{-1}$. Then $g^\iota \in G(\Q), (gX,Y)=(X,g^\iota Y)$ and $\mu(g^\iota)=\mu(g)$. Also, $\Gamma^1_0(p^{2r}D^2)^\iota=\Gamma^1_0(p^{2r}D^2)$ and, since $L^*$ is maximal in the sense of \cite[p. 375]{shimura3}, the set $\{g\in G(\Q)\mid gL^*\subset L^*\}$ is also stable under $\iota$.

Now let $S(\mathfrak{q})=\bigcup_i \Gamma^1_0(p^{2r}D^2)g_i$, and for $X\in L^*$ set
\[
m(\mathfrak{q},X)=|\{i\mid X\in g_i^\iota L^*\}|.
\] 
Note that the the sublattices $g_i^\iota L^*\subset L^*$ are independent of the choice of coset representatives, and that the multiplicities $m(\mathfrak{q},X)$ only depend on the $\Gamma^1_0(p^{2r}D^2)$ orbit of $X$, since $\Gamma^1_0(p^{2r}D^2)$ only permutes the sublattices in question.

Given a fractional ideal $\mathfrak{b}$ of $K$, for convenience write 
\[
\theta_{\mathfrak{b}}(\tau, \mathfrak{z})=y^2\sum_{\substack{X\in L^*\\ X\equiv h_0 (L)}} f^{\mathfrak{b}}_{\tau, \mathfrak{z}}(X),
\]
where $\theta_{\mathfrak{b}}=\theta_{\mathfrak{b},h_0}$. 
As the computations in the rest of the proof are independent of the fixed fractional ideal $\mathfrak{b}$, we drop the notation concerning this choice, which amounts to restricting to the case $\text{Cl}_K=1$. Indeed, all the arguments are of a local nature.

Applying $T(\mathfrak{q})$ to $\overline{\theta(\tau,\mathfrak{z})}$ yields:
\begin{align*}
T(\mathfrak{q})(\overline{\theta})=&\text{N}(\mathfrak{q})^{k-1}\sum_{i}j(g_i,\mathfrak{z})^{-k}y^2\sum_{\substack{X\in L^*\\ X\equiv h_0 (L)}} \overline{f_{\tau, g_i(\mathfrak{z})}(X)}\\
&\text{N}(\mathfrak{q})^{k-1}\sum_{i}\mu(g_i)^{-k}y^2\sum_{\substack{X\in L^*\\ X\equiv h_0 (L)}} \overline{f_{\mu(g_i)^{-1}\tau, \mathfrak{z}}(g_i^\iota X)}\\
&p^{-1}y^2\sum_{\substack{X\in L^*\\ X\equiv h_0 (L)}} m(\mathfrak{q},X) \overline{f_{\text{N}(\mathfrak{q})^{-1}\tau, \mathfrak{z}}(X)},
\end{align*}
where we use the identity
\[
\underline{g(\mathfrak{z})}\eta(g(\mathfrak{g}))^{-1}=\mu(g)^{-1}g\underline{\mathfrak{z}}\eta(\mathfrak{z})^{-1}\overline{j(g,\mathfrak{z})}
\]
and the fact that $\mu(g_i)=\text{N}(\mathfrak{q})$ for all $i$.

Almost all the rest of the proof is devoted to compute the multiplicities $m(\mathfrak{q}, X)$. For each rational prime $q$, let $K_q=K\otimes_\Q \Q_q$, $\mathcal{O}_q=\mathcal{O}_K\otimes_\Z \Z_q$, $V_q=K_q^3$ and $L_q^*=p^{-r}D^{-1}\mathcal{O}_q\oplus \frac{1}{2}\mathcal{O}_q\oplus p^{-r}D^{-1}\mathcal{O}_q$. Also we write $G_q=G(\Q_q)$ and let
\[
\Gamma_q=\{\gamma\in G_q\mid \gamma L_q^*=L_q^*, \gamma h_q=h_q\},
\]
where 
\[
h_q=\begin{cases}
(0,0,-p^{-r})\;\;\; &\text{if } q= p,\\
(0,0,-D^{-1})\;\;\; &\text{if } q\mid D,\\
(0,0,0)\;\;\; &\text{otherwise.}
\end{cases}
\]
Then for a prime $\mathfrak{q}$ of $K$ above $q$, we have
\[
S_\mathfrak{q}=\bigcup_i \Gamma_q g_i
\]
so that $S_{\mathfrak{q}}\subset G_q, \Gamma_qS_{\mathfrak{q}}\Gamma_q=S_{\mathfrak{q}}$ and in fact the above union is still disjoint. For $\ell\ne q$,
\[
(g_i^\iota L)_{\ell}^*=L_{\ell}^*,
\]
so that
\[
X\in g_i^\iota L^*\Leftrightarrow X\in g_i^\iota L_q^*
\]
and 
\[
m(\mathfrak{q},X)=m(q,X)=|\{i\mid X\in g_i^\iota L^*_q\}|,
\]
where the $g_i^\iota$ can be any set of coset representatives for $\Gamma_q$ in $S_\mathfrak{q}$. From this, it immediately follows that if $q$ is coprime to $pD$, then Kudla's computations in \cite[\S 2]{kudla1} apply in the exact same way. Hence we have proved the second part of the theorem, and we can restrict to $q=p$. First of all note that $m(p,X)$ only depends on $X\!\!\mod pL_p^*$, cf. \cite[p. 330]{kudla1}.
We need the following technical result to conclude:
\[
m(p,X)=\begin{cases}
1 &\text{if}\; (X,X)\equiv 0\mod p\\
0 &\text{otherwise}.
\end{cases}
\]
We postpone its proof to the end of the section.
By the above we have,
\[
T(\mathfrak{p})(\overline{\theta})=y^2\sum_{\substack{X\equiv h_0 (L_p)\\p\mid (X,X)}}\overline{f_{p^{-1}\tau, \mathfrak{z}}(X)}.
\]
On the other hand,
\[
\overline{T(p)(\theta)}=p^{-1}\sum_{j=0}^{p-1}\overline{\theta(p^{-1}(\tau+j),\mathfrak{z})}=p^{-2}y^2\sum_{\substack{X\equiv h_0 (L_p)\\p\mid (X,X)}}\overline{f_{p^{-1}\tau, \mathfrak{z}}(X)},
\]
and the result follows easily by the usual adjointness properties of the Hecke operators with respect to the Petersson inner product, cf. \cite[p. 327]{kudla1}; note that the dependence on $\varepsilon$ of the final formula is analogous to the result in Finis \cite[p. 11]{finis} and is due to the presence of the factor $\varepsilon(\overline{\det g})$ in the definition of $\theta(\tau, g)$.
\end{proof}

Now we only need to prove the technical lemma:
\begin{Lemma}
If $X\in L_p^*, X\equiv h_0\mod L_p$, is such that $(X,X)\equiv 0\!\! \mod p$, then $m(p,X)=1$, otherwise $m(p,X)=0$.
\end{Lemma}

\begin{proof}
First of all, note that for $X$ as in the statement, $(X,X)\in \Z$. Hence, if $Y=g_i^\iota X$ for some $X$ as above, then $(Y,Y)=\mu_{g_i^\iota}(X,X)\equiv 0\mod p$, which proves the second part of the lemma. We may thus suppose that $p\mid (X,X)$. Then, it is clear that the function $m(p,X)$ is actually defined on $L_p^*/pL_p^*$, and is constant on $\Gamma_p$-orbits in this set. Therefore to compute $m(p,X)$ it will be sufficient to find a set of $\Gamma_p$-orbit representatives and a set of coset representatives $g_j$; finally, for each such orbit representative we will count the number of coset representatives which contain the former in their image.

Recall that $p$ splits in $K$, and so take idempotents $e$ and $e^\sigma$ in $K_p$ corresponding to the two primes dividing $p$. Then
\[
V_p=eV_p+e^\sigma V_p\cong \Q_p^3\times \Q_p^3
\]
and
\[
L_p^*=eL_p^*+e^\sigma L_p^*\cong (p^{-r}\Z_p+\Z_p+p^{-r}\Z_p)\times (p^{-r}\Z_p+\Z_p+p^{-r}\Z_p).
\]
Set $J_1=eJ, J_2=e^\sigma J$. Then
\begin{align*}
G_p=G(\Q_p)\cong & \{(g_1,g_2)\in \GL_3(\Q_p)\times \GL_3(\Q_p)\mid {^t}g_1J_2g_2=cJ_2, c\in \Q_p^\times\}\\
\cong &\GL_3(\Q_p)\times \Q_p^\times,
\end{align*}
since $g_2=cg_1^*$ with $g_1^*=J_2^{-1}{^t}g_1^{-1}J_2$. The action of $[g,c]\in \GL_3(\Q_p)\times \Q_p^\times$ on $[X,Y]\in \Q_p^3\times \Q_p^3$ is given by
\[
[g,c][X,Y]=[gX,cg^*Y].
\]
The above action is equivalent to the standard action
\[
[g,c]'[X,Y]=[gX,c{^t}g^{-1}Y]
\]
under the isomorphism $1\times J_2:\Q_p^3\times \Q_p^3\to \Q_p^3\times \Q_p^3$. Given a set of representatives $\tilde{g}_i$ for $S_{\overline{\mathfrak{p}}}/\Gamma^1_0(p^{2r}D^2)_p$, since $S_{\overline{\mathfrak{p}}}=S_{\mathfrak{p}}^\iota$, we know that we may take $g_i^\iota=\tilde{g}_i$.
By \cite[p. 333]{kudla1}, the images of these $g_i^\iota$'s in $\GL_3(\Q_p)\times \GL_3(\Q_p)$ under the isomorphism corresponding to the above standard action will be, cf. \cite[p. 334]{kudla1},
\[
\Bigg[\begin{pmatrix}
p&\mu b&a \\
&1&\\
&& 1
\end{pmatrix},\begin{pmatrix}
1&&\\
-\mu b&p&\\
a&&p
\end{pmatrix}\Bigg],
\]
where $a$ and $b$ run over a set of representatives for $\Z_p/p\Z_p$ and the image of $\delta$ in $\Q_p\times \Q_p$ is $[\mu,-\mu]$. 

The reduction modulo $p$ of these elements gives representatives in $M_3(\F_p)\times M_3(\F_p)$:
\[
\Bigg[\begin{pmatrix}
0&\mu b&a \\
&1&\\
&& 1
\end{pmatrix},\begin{pmatrix}
1&&\\
-\mu b&0&\\
a&&0
\end{pmatrix}\Bigg],
\]
Since we consider $X\equiv h_0\mod L$, we easily see that each such vector is in the image of exactly $1$ representative, which implies the result.
\end{proof}

\section{The $p$-adic interpolation of the Kudla lift}
In this section, we write $\mathcal{L}$ for $\mathcal{L}^{p,\text{ar}}$, and use the same convention for the Fourier-Jacobi coefficients, to simplify the notation.
In the following, we show that the lift defined in section \ref{kudla} induces a lift of $\Lambda$-adic elliptic modular forms to $\Lambda$-adic Picard modular forms. First of all, we recall the notions of $p$-adic and $\Lambda$-adic modular forms both in the elliptic and in the Picard case, following \cite[\S 3.3]{hsieh} and \cite[\S 5.3]{goren} for the latter.

A formal $q$-expansion $f\in \overline{\Z}_p[[q]]$ is a $p$-adic modular form (resp. cusp form) of tame level $\Gamma_1(D)$, if it is the $p$-adic limit of classical modular forms $f_i$ of tame level $\Gamma_1(D)$. As is well known, the weights $k_i$ of the modular forms $f_i$ also converge to a character $\kappa\in \text{Hom}_{\text{cont}}(\Z_p^\times,\Z_p^\times)$. We will write $f\in M^{Serre}_{\kappa-1}(\Gamma_1(D),\overline{\Z}_p)$ (resp. $S^{Serre}_{\kappa-1}(\Gamma_1(D),\overline{\Z}_p)$), to underline that it is a $p$-adic modular form (resp. cusp form) in the sense of Serre.

\subsection{On $\Lambda$-adic elliptic cusp forms}
We start by reviewing Hida's construction of elliptic cuspidal $\Lambda$-adic modular forms of tame level $D$. Let $\Gamma=1+p\Z_p$ be the maximal torsion-free subgroup of $\Z_p^\times$. We fix a topological generator $\gamma\in \Gamma$ such that $\Gamma=\overline{\langle\gamma\rangle}$. Let $\Lambda=\Z_p[[\Gamma]]$ and $\Lambda_1=\Z_p[[\Z_p^\times]]$ be the completed group rings on $\Gamma$ and on $\Z_p^\times$ over $\Z_p$, respectively. Clearly $\Lambda_1$ has a natural $\Lambda$-algebra structure induced from $\Lambda_1\cong \Lambda[\mu_{p-1}]$, where $\mu_{p-1}$ denotes the maximal torsion subgroup of $\Z_p^\times$ consisting of $(p-1)$-th roots of unity.

\begin{definition}[$\Lambda$-adic weight spaces]
\label{weigths}
For each $\Lambda_1$-algebra $R$ finite flat over $\Lambda$, we define the $\Lambda$-adic weight space $\mathfrak{X}(R)$ associated with $R$ as
\[
\mathfrak{X}(R):=\text{Hom}_{\text{cont}}(R,\overline{\Q}_p),
\]
on which the following arithmetic notions are introduced:
\begin{enumerate}[(i)]
\item A point $P\in \mathfrak{X}(R)$ is said to be arithmetic if there exists an integer $k\geq 2$ such that the restriction of $P$ to $\mathfrak{X}(\Lambda)\cong \Hom_{\text{cont}}(\Gamma,\overline{\Q}_p^\times)$ corresponds to a continuous character $P_{k,\epsilon}:\Gamma\to \overline{\Q}_p^\times$ satisfying $P_{k,\epsilon}(\gamma)=\epsilon(\gamma)\gamma^k$, for some finite order character $\epsilon:\Z_p^\times\to \overline{\Q}_p^\times$. We denote by $\mathfrak{X}_{\text{alg}}(R)$ the set of all arithmetic points in $\mathfrak{X}(R)$.
\item An arithmetic point $P\in \mathfrak{X}_{\text{alg}}(R)$ is said to have signature $(k,\epsilon)$ if there exists an integer $k\geq 2$ and a finite character $\epsilon:\Z_p^\times\to \overline{\Q}_p^\times$ such that $P$ lies over the point $P_{k,\epsilon}\in \mathfrak{X}_{\text{alg}}(\Lambda_1)\cong \Hom_{\text{cont}}(\Z_p^\times,\overline{\Q}_p^\times)$ corresponding to the character $P_{k,\epsilon}(y)=y^k\epsilon(y)$ on $\Z_p^\times$. For simplicity, we just write $P=(k,\epsilon)$ and often refer to it as the arithmetic point of weight $k$ and Nebentypus $\epsilon\omega^{-k}$, where $\omega$ denotes the Teichm\"uller character.
\end{enumerate}
\end{definition} 
We note that $\mathfrak{X}(\Lambda)$ has a natural analytic structure induced from the identification $\Hom_{\text{cont}}(\Lambda,\overline{\Q}_p)\cong \Hom_{\text{cont}}(\Gamma,\overline{\Q}_p^\times)$. Moreover, for a given $R$ the natural restriction to $\Lambda$ gives a finite-to-one mapping
\[
\pi:\mathfrak{X}(R)\twoheadrightarrow\mathfrak{X}(\Lambda_1)\twoheadrightarrow\mathfrak{X}(\Lambda),
\]
which allows us to define some analytic charts around all points of $\mathfrak{X}_{\text{alg}}(R)$. Indeed, it is established by Hida that each $P\in \mathfrak{X}_{\text{alg}}(R)$ is unramified over $\mathfrak{X}(\Lambda)$, and consequently there exists a natural section of $\pi$
\[
s_P:U_P\subseteq\mathfrak{X}(\Lambda)\to \mathfrak{X}(R)
\]
defined on a neighborhood $U_P$ of $\pi(P)$. These local sections endow $\mathfrak{X}(R)$ with analytic charts around points in $\mathfrak{X}_{\text{alg}}(R)$. For each $P\in \mathfrak{X}_{\text{alg}}(R)$, a function $\textbf{f}:\mathfrak{U}\to \overline{\Q}_p$ defined on $\mathfrak{U}=s_P(U_P)\subset\mathfrak{X}(R)$ is said to be analytic if $\textbf{f}\circ s_P:U_P\to \overline{\Q}_p$ is analytic. In parallel, an open subset $\mathfrak{U}\subset \mathfrak{X}(R)$ containing some $P\in \mathfrak{X}_{\text{alg}}(R)$ is called an analytic neighborhood of $P$ if $\mathfrak{U}=s_P(U_P)$. For instance, we easily see that each element $\textbf{a}\in R$ gives rise to an analytic function $\textbf{a}:\mathfrak{X}_{\text{alg}}(R)\to \overline{\Q}_p$ defined by $\textbf{a}(P)=P(\textbf{a})$. In general, if $P\in \mathfrak{X}(R)$ unramified over $\mathfrak{X}(\Lambda)$, then each element $\textbf{a}\in R_{(P)}$ gives rise to an analytic function defined on some analytic neighborhood of $P$, where $R_{(P)}$ denotes the localization of $R$ at $P$, and it gives rise to a discrete valuation ring finite and unramified over $\Lambda$. According to the standard usage, we refer to the evaluation $\textbf{a}(P)$ at $P\in \mathfrak{X}_{\text{alg}}(R)$ as the specialization of $\textbf{a}$ at $P$.

For each integer $r\geq 1$, let $$X_1(p^{2r}D^2)=\Gamma_1^0(p^{2r}D^2)\backslash(\h\cup \mathbb{P}^1(\Q))$$ be the compactified modular curve, and $V_r=H^1(X_1(p^rD),\Z_p)$ the singular cohomology group of $X_1(p^rD)$ with values in $\Z_p$. It is well-known that $V_r$ is canonically isomorphic to the parabolic cohomology group $H^1_{\text{par}}(\Gamma_1(p^rD),\Z_p)\subset H^1(\Gamma_1(p^rD),\Z_p)$, which is defined to be the image of the compactly supported cohomology group under the natural map, cf. \cite[p. 358]{hida}. We denote the abstract $\Lambda$-adic Hecke algebra of tame level $D$ by the free polynomial algebra
\[
\mathbb{T}:=\Lambda_1[T_n\mid 1\le n\in \Z]
\]
generated by the Hecke operators $T_n$ over $\Lambda_1$. Since $\mathbb{T}\cong \Z_p[T_n, \Z_p^\times]$, a natural action of $\mathbb{T}$ on $V_r$ is defined by regarding the generator $T_n$ act via the $n$-th Hecke correspondence and elements of $\Z_p^\times$ act via the usual action of the diamond operators. For each pair of positive integer $r_1\geq r_2$ we get a corestriction map $V_{r_1}\to V_{r_2}$, induced by $\Gamma_1(p^{r_1})\hookrightarrow\Gamma_1(p^{r_2})$, which commutes with the action of $\mathbb{T}$. Hence the projective limit
\[
V_\infty:=\lim_{\substack{\leftarrow\\ r\geq 1}}V_r
\]
is naturally endowed with a $\mathbb{T}$-algebra structure. We denote by $V_\infty^{\text{ord}}$ the direct factor of $V_\infty$ cut out by Hida's idempotent $e_{\text{ord}}=\lim_{m\to \infty}U_p^{m!}$, that is,
\[
V_\infty^{\text{ord}}=e_{\text{ord}}\cdot V_\infty,
\]
on which $U_p$ acts invertibly. We note that $V_\infty^{\text{ord}}$ is a $\Lambda$-algebra free of finite rank \cite[p. 209, Theorem 1]{hida}. Moreover, for $\mathscr{L}=\text{Frac}(\Lambda)$, Hida constructed an idempotent $e_{\text{prim}}$ in the image of $\mathbb{T}\otimes_\Lambda \mathscr{L}$ in $\text{End}_{\mathscr{L}}(V_\infty^{\text{ord}}\otimes_\Lambda \mathscr{L})$, which can be regarded as an analogue of the projection to the space of primitive Hecke eigenforms in Atkin-Lehner theory, \cite[p. 250-252]{hidaiwa}. Then we define the universal ordinary parabolic cohomology group of tame level $D$ by the $\mathbb{T}$-agebra
\[
\mathbb{V}^{\text{ord}}:=V_\infty^{\text{ord}}\cap e_{\text{prim}}(V_\infty^{\text{ord}}\otimes_\Lambda \mathscr{L}),
\]
which is a reflexive $\Lambda$-algebra of finite rank and is consequently locally free of finite rank over $\Lambda$. Then the universal $p$-ordinary Hecke algebra of tame level $D$ is defined to be the image $R^{\text{ord}}$ of the homomorphism
\[
h:\mathbb{T}\to\text{End}_{\Lambda_1}(\mathbb{V}^{\text{ord}}).
\] 
Each morphism $\phi:\mathbb{T}^{\text{ord}}\to \Lambda_1$ is naturally equipped with a formal $q$-expansion
\[
\textbf{f}^{\,\text{ord}}_\phi=\sum_{n\geq 1}\textbf{a}_nq^n\in \Lambda_1 [[q]],\;\;\; \textbf{a}_n=\phi(T_n),
\]
which is called a universal $p$-stabilized ordinary form of tame level $D$.

\begin{definition}[Ordinary $p$-stabilized newforms]
For given integers $k\geq 2$ and $r\geq 1$, a Hecke eigenform $f_k^*\in S_k(\Gamma_1(p^rD),\omega_{K/\Q})$ is called an ordinary $p$-stabilized newform if one of the following conditions holds true:
\begin{enumerate}[(i)]
\item $f_k^*$ is a $p$-ordinary Hecke eigenform in $S_k^{\text{new}}(\Gamma_1(p^rD),\omega_{K/\Q})$, where the latter is the subspace of $S_k(\Gamma_1(p^rD),\omega_{K/\Q})$ consisting of newforms.
\item If $r=1$, then there exists a normalized ordinary Hecke eigenform $f_k\in S_k(\Gamma_1(D),\omega_{K/\Q})$ such that 
\[
f_k^*(z)=f_k(z)-\beta_p(f_k)f_k(pz),
\]
where $\beta_p(f_k)$ is the non-unit root of the polynomial $X^2-a_p(f_k)X+p^{k-1}$.
\end{enumerate}
\end{definition}
\begin{Remark}
It follows from the definition that the ordinary $p$-stabilized newforms are literally $p$-ordinary Hecke eigenforms. Indeed, this assertion is trivial in the case of $(i)$. If $f_k^*\in S_k(\Gamma_1(pD),\omega_{K/\Q})$ is taken as in $(ii)$, then for each prime $\ell$, we have
\[
a_\ell(f_k^*)=\begin{cases}
a_\ell(f_k) & \text{if}\;\; \ell \ne p,\\
\alpha_p(f_k) & \text{if}\;\; \ell = p,
\end{cases}
\]
where $\alpha_(f_k)$ is the unit root of the polynomial $X^2-a_p(f_k)X+p^{k-1}=(X-\alpha_p(f_k))(X-\beta_p(f_k))$. For a given $p$-ordinary Hecke eigenform in $S_k(\Gamma_1(D),\omega_{K/\Q})$, this type of $p$-adic normalization process is called the ordinary $p$-stabilization. However, note that each ordinary $p$-stabilized newform $f_k^*\in S_k(\Gamma_1(pD),\omega_{K/\Q})$ is actually an oldform except for $k=2.$ Indeed, if $f_k^*\in S_k^{\text{new}}(\Gamma_1(pD))$, then we have
\[
a_p(f_k^*)=\pm p^{k/2-1}.
\]
\end{Remark}

\noindent The following theorem has been established by Hida:

\begin{Theorem}\emph{\cite[Theorem 2.6]{greenberg}}
Let $r$ be a fixed positive integer and $\textbf{f}^{\,\text{ord}}=\sum_{n\geq 1}\textbf{a}_n q^n\in \Lambda_1[[q]]$ be a universal ordinary $p$-stabilized form of tame level $D$ and character $\omega_{K/\Q}$, in the sense of \cite[p. 91]{hidashimura}. Then for each $P\in \mathfrak{X}_{\text{alg}}(\Lambda_1)$, the specialization
\[
\textbf{f}^{\,\text{ord}}(P)=\sum_{n\geq 1}\textbf{a}_n(P) q^n\in \overline{\Q}_p[[q]]
\]
induces a one-to-one correspondence
\begin{align*}
\{P=&(k,\epsilon)\in \mathfrak{X}_{\text{alg}}(\Lambda_1)\mid 2\le k\in \Z,\epsilon:\Z_p^\times \to \overline{\Q}_p^\times \text{ character of order } p^r\}\\
&\longleftrightarrow \{f_k^*\in S_k(\Gamma_1(p^rD),\epsilon\omega^{-k}\omega_{K/\Q})\}\mid \text{ordinary}\;\; p\text{-stabilized newform}\}.
\end{align*}
\end{Theorem}

By applying the above theorem for $r=1$, each $P=(2k,\omega^{2k})\in \mathfrak{X}_{\text{alg}}(\Lambda_1)$ corresponds to an ordinary $p$-stabilized newform $\textbf{f}_{\text{ord}}(P)=f_{2k}^*\in S_{2k}(\Gamma_0(pD),\omega_{K/\Q})$ associated with a $p$-ordinary normalized Hecke eigenform $f_{2k}\in S_{2k}(\Gamma_1(D),\omega_{K/\Q})$ via the ordinary $p$-stabilization. Therefore, for a fixed $P_0=(2k_0,\omega^{2k_0})\in \mathfrak{X}_{\text{alg}}(\Lambda_1)$ with $k_0>\!\!>0$, we may regard $\textbf{f}_{\text{ord}}\in \Lambda_1[[q]]$ as a $\Lambda$-adic elliptic cusp form, and we consequently obtain a $p$-adic analytic family of ordinary $p$-stabilized newforms $\{\textbf{f}_{\text{ord}}(P)=f_{2k}^*\}$ parametrized by $P=(2k,\omega^{2k})\in \mathfrak{X}_{\text{alg}}(\Lambda_1)$ with $k_0\le k\in \Z$. 

\subsection{On $\Lambda$-adic Picard modular forms}
In the following, write $K_p=K\otimes \Q_p$ and $\mathcal{O}_p=\mathcal{O}_K\otimes \Z_p$.
\begin{definition}
Let 
\[
\mathfrak{X}_p=\lim_{\substack{\longleftarrow\\m}} [\Z/(p-1)p^{m-1}\Z\times \Z/(p-1)p^{m-1}\Z]
\]
be the space of $p$-adic weights.
\end{definition}
Recall the embedding $\overline{\sigma}:K\hookrightarrow \Cc$. If $\kappa\in \mathfrak{X}_p$, then $\overline{\sigma}^\kappa$ is a well-defined locally $\Q_p$-analytic homomorphism of $\mathcal{O}_p^\times$ to itself, but note that not every such homomorphism is a $\overline{\sigma}^\kappa$ for some $\kappa\in \mathfrak{X}_p$. 
For $\kappa\in \mathfrak{X}_p$, a $p$-adic Picard modular form à la Serre is defined as a formal Fourier-Jacobi expansion which is the limit of $q$-expansions of classical Picard modular forms of integral weight, cf. \cite[p. 60]{goren} for the analogous definition in the inert case. For a tame level $N$, we denote the space of $p$-adic Picard modular forms of level $N$ and weight $\kappa\in \mathfrak{X}_p$ as $\mathscr{M}^{Serre}(\kappa,N, K_p)$. There is another possible construction of $p$-adic modular forms, which is more geometric, à la Katz. Here is an overview of the construction: first of all, recall that as $p$ is split in $K$, the ordinary locus of the Picard modular surface is open and dense \cite[Lemma 8.10]{hidashimura}. Let 
\[
V_r^{(m)}=H^0(\overline{Ig}(p^r)^{(m)},\mathcal{O}_{\overline{Ig}(p^r)^{(m)}})
\]
be the ring of regular functions on the truncated Igusa scheme of level $p^r$ over the completion of the ordinary locus, $\overline{Ig}(p^r)^{(m)}$. Here $\mathcal{O}_{\overline{Ig}(p^r)^{(m)}}$ is the sheaf of regular functions on $\overline{Ig}(p^r)$. The latter object's definition is detailed in \cite[p. 777]{hsieh}, but let us just recall that this is the étale cover of the truncation modulo $p^m$ of the ordinary locus which classifies ordinary points on the Picard scheme, together with a trivialization of their group of $p^r$-torsion points. Now we also let
\[
V^{(m)}=\lim_{\substack{\rightarrow\\r}}V_r^{(m)},\;\;\; V=\lim_{_{\substack{\leftarrow\\m}}} V^{(m)}.
\]
We call $V$ the space of Katz $p$-adic modular forms of all weights. Recall that on $\overline{Ig}(p^r)^{(m)}$, we have an action of $\Delta(p^r)=(\mathcal{O}_K/p^r\mathcal{O}_K)^\times$ via
\[
\gamma(\underline{A},\epsilon^{(m)}_r)=(\underline{A},\epsilon^{(m)}_r\circ \gamma^{-1}).
\]
Consequently we get actions of $\Delta(p^r)$ and $\Delta:=\lim_{\leftarrow}\Delta(p^r)$ on $V^{(m)}$ and on $V$, respectively, defined by $\gamma(f)=f\circ \gamma^{-1}$. Let $\kappa\in \mathfrak{X}_p$ and define, for a fixed tame level $N$,
\[
\mathscr{M}^{Katz}(\kappa,N,\mathcal{O}_p)=V(\overline{\sigma}^\kappa)=\{f\in V\mid \gamma(f)=\overline{\sigma}^\kappa(\gamma)\cdot f\;\; \forall\, \gamma\in \Delta\}.
\]
\begin{Theorem}
\label{15}
For $\kappa\in \mathfrak{X}_p$, there is a natural isomorphism
\[
\mathscr{M}^{Serre}(\kappa,N, \mathcal{O}_p)\cong \mathscr{M}^{Katz}(\kappa,N, \mathcal{O}_p).
\]
\end{Theorem}

\begin{proof}
This is the extension of the result \cite[Proposition 5.8]{goren} formulated for $p$ inert, by the $q$-expansion principle for $p$ split, proved in \cite[Theorem 3.11]{hsieh}. Indeed, the proof of \cite[Proposition 5.8]{goren} is independent of the behavior of $p$ in $K$.
\end{proof}

Analogously, a $p$-adic cusp form is defined as a limit of classical Picard cusp forms, and Theorem \ref{15} restricts to cuspidal forms. Later, we will usually drop $\mathcal{O}_p$ from the notation, and we will consider, for $\chi_\kappa$ a $p$-adic Hecke character of weight $\kappa$, the spaces
\[
\mathscr{S}(\kappa,N, \chi_\kappa)\subset \mathscr{M}(\kappa,N, \chi_\kappa)
\]
of $p$-adic Picard modular forms of central character $\chi_\kappa$, meaning that the limit is taken on Picard modular forms of central characters $\chi_{k_i}$ converging to $\chi_{\kappa}$ $p$-adically.

We introduce the notion of $\Lambda$-adic Picard modular forms from the point of view of Fourier expansions. Recall that in section \ref{arithmetic} we introduced the notions of arithmetic theta functions, following Shimura. 

\begin{definition}[$\Lambda$-adic theta function]
For $R$ a $\Lambda_1$-algebra, finite flat over $\Lambda$, define the space of $\Lambda$-adic theta functions $\mathcal{T}_{d}(R)$ of degree $d$ simply as
\[
\mathcal{T}_{d}^{\emph{ar}}\otimes_{\overline{\Q}_p} R.
\]
Furthermore, write
\[
\mathcal{T}(R)=\mathcal{T}^{\emph{ar}}\otimes_{\overline{\Q}_p} R,
\]
for the space of theta functions of any degree.
If $\textbf{g}\in \mathcal{T}_{d}(R)$  and $P\in \mathfrak{X}(R)$ then define $\textbf{g}(P)\in \mathcal{T}_{d}^{\emph{ar}}$ by applying $P$ on the second component, that is, if $\textbf{g}=\theta\otimes _{\overline{\Q}_p} \textbf{a}$, then
\[
\textbf{g}(P)=\theta\otimes _{\overline{\Q}_p} \textbf{a}(P).
\]
\end{definition}

\begin{Remark}
In greater generality, since we have defined our $p$-adic weight space as $\mathfrak{X}_p=(\Z/(p-1)\Z\times \Z_p)^2$, we could consider the bigger space $\mathcal{O}_{\mathfrak{X}_p}$ instead of $\Lambda_1$, but for our purposes this will suffice. Indeed we have the natural inclusions $$\Lambda\hookrightarrow \Lambda_1\hookrightarrow \mathcal{O}_{\mathfrak{X}_p}.$$
\end{Remark}
\begin{definition}[$\Lambda$-adic Picard modular form]
Let $R$ be a $\Lambda_1$-algebra, finite flat over $\Lambda$. Pick $P_0=(k_0,\omega^{k_0})\in \mathfrak{X}_{\emph{alg}}(R)$. A formal Fourier-Jacobi expansion
\[
\textbf{F}=(\textbf{F}_{\mathfrak{a}})_{\mathfrak{a}\in \mathcal{A}}=\Big(\sum_{n\geq 0} \textbf{g}_{n/\emph{N}(\mathfrak{a}),\mathfrak{a}}(w)q^n\Big)_{\mathfrak{a}\in \mathcal{A}}\in \mathcal{T}(R_{(P_0)})[[q]]
\]
is called a $\Lambda$-adic Picard modular form of tame level $D$ if there exists an analytic neighborhood $\mathfrak{U}_0$ of $P_0$ such that for each arithmetic point $P=(k,\epsilon\omega^k)\in \mathfrak{U}_0$ with $k\geq k_0$ and $\epsilon:\Z_p^\times\to \overline{\Q}_p^\times$ a finite character, the specialization
\[
\textbf{F}(P):=\Big(\sum_{n\geq 0} \textbf{g}_{n/\emph{N}(\mathfrak{a}),\mathfrak{a}}(w)(P)q^n\Big)_{\mathfrak{a}}\in \overline{\Q}_p[[q]]
\]
gives rise to the Fourier expansion of a holomorphic Picard modular form in $\mathscr{M}(k,\Gamma_1(p^rD))$, for some $r$. In particular, a $\Lambda$-adic Picard modular form $\textbf{F}$ is said to be cuspidal if $\textbf{F}\in \mathscr{S}(k,\Gamma_1(p^rD))$ for almost all $P\in \mathfrak{U}_0$.
\end{definition}

If there exists a $\Lambda$-adic Picard modular form $\textbf{F}\in \mathcal{T}(R_{(P_0)})[[q]]$, then every coefficient $\textbf{g}_{n/\emph{N}(\mathfrak{a}),\mathfrak{a}}(w)\in \mathcal{T}_{n/\emph{N}(\mathfrak{a})}(R_{(P_0)})$ of $\textbf{F}$ gives rise to an analytic function $\mathfrak{U}_0\to \mathcal{T}_{n/\emph{N}(\mathfrak{a})}^{\textit{ar}}$. Hence every specialization $\textbf{F}(P)$ gives a holomorphic Picard modular form whose Fourier-Jacobi coefficients are $p$-adic analytic functions on $\mathfrak{U}_0$, cf. \S5.1 for the analogous elliptic case.

\begin{Remark}
Technically we should call $\mathcal{T}_{d}(R)$ the space of $R$-adic theta functions, and accordingly define $R$-adic Picard modular forms, but in the following we allow this slight abuse of notation.
\end{Remark}

\subsection{The $p$-adic lift}
For a $p$-adic elliptic modular form $f$, let $\mathcal{L}^*(f)$ be the formal series with Fourier-Jacobi coefficients
\[
g_{n/N(\mathfrak{a}),\mathfrak{a}}^*(w)=\kappa(\Omega_0)^{-1}\varepsilon(\overline{\mathfrak{a}}/\mathfrak{a})\sum_{\substack{\mathfrak{b}\in \text{Cl}_{K}\\ (\mathfrak{b},p)=1}}(\varepsilon^3\chi_\kappa^{-1})(\mathfrak{b})\kappa(\text{N}(\mathfrak{b}))^{1/2}g_{n/N(\mathfrak{a}),\mathfrak{a},\mathfrak{b}}^*(w),
\]
where
\[
g_{n/N(\mathfrak{a}),\mathfrak{a},\mathfrak{b}}^*(w)=T_{n}^*(f(\tau)\vartheta_{\overline{\mathfrak{b}}}(n\delta w,\tau))|_{\tau=-\overline{\tau}_{\mathfrak{a},\mathfrak{b}}},
\]
and $\kappa(\Omega_0)$ is the $p$-adic period associated to the $p$-adic weight $\kappa$ by \cite[p. 436.]{greenberg}



\begin{Proposition}
\label{prop6}
For $\kappa$ even, the rule $f\mapsto \mathcal{L}^*(f)$ defines a morphism of Hecke modules
\[
\mathcal{L}^*:S^{Serre}_{\kappa-1}(\Gamma_1(D),\omega_{K/\Q},\overline{\Z}_p)\to \mathscr{S}^{Serre}(\kappa,\Gamma_1(D),\chi_\kappa).
\]
\end{Proposition}

We postpone the proof of this proposition to the subsection \ref{5.4}.

\subsection{The $\Lambda$-adic lift}
\label{5.4}
Now we fix $k_0\geq 6$ an even integer, and suppose that $f_0\in S_{k_0-1}(\Gamma_1(D),\omega_{K/\Q},\overline{\Z}_p)$ is a normalized $T_p$-ordinary Hecke eigenform. By the discussion in the previous section, its ordinary $p$-stabilization $f_0^*$ is a Hecke eigenform with the same eigenvalues of $f$ at all primes except $p$, where $f_0^*$ has eigenvalue $\alpha_p$, the latter being the unique root of the polynomial $X^2-a_pX+p^{k_0-2}$ which is a $p$-adic unit. Additionally, we know that $f_0^*$ is part of a Hida family $\mathfrak{F}$, which means that $P_{k_0-1,\omega^{k_0-1}}\circ \mathfrak{F}=f_0^*$. We want to define a $\Lambda$-adic lift of $\mathfrak{F}$, i.e. a formal Fourier-Jacobi expansion with coefficients in $\Lambda$, whose specializations at an arithmetic prime $P_k$ is the lift of the specialization of $\mathfrak{F}$ at $P_{k-1}$. Concretely, for each $j\in \Z/pZ$ and for any  $k_0\le k\equiv j\mod (p-1)$, we define the $\Lambda$-adic Fourier-Jacobi expansion $\mathfrak{L}_{j}(\mathfrak{F})$ by
\begin{equation}
\label{lambda}
\mathfrak{g}_{n/\text{N}(\mathfrak{a}),\mathfrak{a}}^{j}(w)=\varepsilon(\overline{\mathfrak{a}}/\mathfrak{a})\sum_{\substack{\mathfrak{b}\in \text{Cl}_{K}\\ (\mathfrak{b},p)=1}}(\varepsilon^3\Xi^{-1}\omega_{\mathfrak{p}}^{-1}\omega_{\overline{\mathfrak{p}}}^{j-k_0-1})(\mathfrak{b})\mathfrak{g}_{n/\text{N}(\mathfrak{a}),\mathfrak{a},\mathfrak{b}}(w),
\end{equation}
where
\[
\mathfrak{g}_{n/\text{N}(\mathfrak{a}),\mathfrak{a},\mathfrak{b}}(w)=T_{n}^*(\mathfrak{F}(\tau)\star\vartheta_{\overline{\mathfrak{b}}}(n\delta w,\tau))|_{\tau=-\overline{\tau}_{\mathfrak{a},\mathfrak{b}}},
\] 
$T_{n}^*$ now are the adjoint universal Hecke operators \cite[p. 209]{hida}, \cite[Theorem 2.1.14]{ohta}, acting on the space of $\Lambda$-adic modular forms and $\star$ is the convolution introduced by Hida, cf. \cite[p. 200]{hida}.
Recall also that by \cite[p. 93]{hidashimura}, it is possible to evaluate a given ordinary $\Lambda$-adic form at a CM-point: indeed, the underlying elliptic curve is ordinary, as $p$ splits in $K$.
Finally, $\Xi$ is a $\Lambda$-adic character which interpolates the Hecke characters $\chi_k|\cdot|_K^{-k/2}$ as $k$ varies. To define this, we adapt the strategy of Collins in \cite[Lemma 5.2.1]{collins}:
\begin{Lemma}
Let $k_0$ be a positive even integer as above and $\chi_{k_0}:\mathbb{I}_K/K^\times\to \Z_{p^{e}}^\times$ algebraic Hecke character for $K$ of weight $k_0$, where $\Z_{p^e}$ is the ring of integers of a finite extension $F/\Q_p$. Then there exists a family of characters $\Xi$ with values in $\mathcal{O}_F[[X]]$ such that if we set $\xi_{k}=P_{k}\circ \Xi$ for a positive even integer, we have:
\begin{enumerate}[(i)]
\label{hecke}
\item For every $k$, $\xi_k$ is a character of infinity type $(0,-k)$.
\item $\xi_{k_0}=\chi_{k_0}|\cdot|_K^{-k_0/2}$.
\item For $k\equiv k_0\mod (p-1)$, $\xi_k$ has the same conductor and same finite type as $\chi_{k_0}$; for $k\not \equiv k_0\mod (p-1)$ the finite types differ by powers of the Teichm\"uller character $\omega_{\overline{\mathfrak{p}}}$.
\end{enumerate}
\end{Lemma}

The proof of this lemma will follow once we have defined a certain $p$-adic Hecke character, and the family passing through it.
Let $\alpha:\mathbb{I}_{K}/K^\times\to \overline{\Q}_p^\times$ be the algebraic Hecke character such that, for $x=(x_v)\in \prod \mathcal{O}^\times_{K_v}$, $\alpha(x)=x_{\mathfrak{p}}^{-1}\omega_{\mathfrak{p}}(x_{\mathfrak{p}})$, where $\omega_{\mathfrak{p}}:(\mathcal{O}_{K}/\mathfrak{p})^\times\to \Z_p^\times$ is the canonical Teichm\"uller character for $\mathcal{O}_{K}$. The character $\alpha$ has infinity type $(1,0)$, conductor $\mathfrak{p}$ and finite type $\omega^{-1}_{\mathfrak{p}}$ (this is a consequence of the convention adopted), see \cite[\S 5.2]{collins} for details.

Remark that if we use $\overline{\mathfrak{p}}$ instead of $\mathfrak{p}$, we get a character of infinity type $(0,1)$, conductor $\overline{\mathfrak{p}}$ and finite type $\omega_{\overline{\mathfrak{p}}}^{-1}$. Indeed, this is simply $\overline{\alpha}:=\alpha\circ \tilde{c}$, where $\tilde{c}:\mathbb{I}_{K}/K^\times\to \mathbb{I}_{K}/K^\times$ is the automorphism on ideles induced by the non-trivial automorphism $c$ of $K$. Clearly $\overline{\alpha}(\mathfrak{a})=\alpha(\overline{\mathfrak{a}})$.

We may now use the character $\overline{\alpha}$ to define a family of characters. In general, if we are given a character $\mathbb{I}_{K}/K^\times\to \Gamma$ for $\Gamma$ the $p$-profinite group of units that are congruent to $1$ modulo the uniformizer, i.e. principal units, in a finite extension $L$ of $\Q_p$, we can just define a character $\mathbb{I}_{K}/K^\times\to \mathcal{O}_L[[\Gamma]]^\times$ by composing with the tautological embedding $\iota:\Gamma\hookrightarrow \mathcal{O}_L[[\Gamma]]^\times$. Now recall that we have $\overline{\alpha}:\mathbb{I}_{K}/K^\times\to (1+\pi)^{\Z_p}\subset \mathcal{O}_F^\times$ for some finite extension $F/\Q_p$, with $(1+\pi)^{p^e}=(1+p)$ for some positive integer $e$, and hence take $\mathcal{I}=\mathcal{O}_F[[\Gamma']]$ for $\Gamma'=(1+\pi)^{\Z_p}$. 
We thus define the character $\mathcal{A}:=\iota\circ \overline{\alpha}$, which is a continuous $\mathcal{I}$-adic character, by construction. 

\begin{proof}[Proof of the lemma]
We set
\[
\Xi:=\chi_{k_0}|\cdot|_K^{-k_0/2}\overline{\alpha}^{k_0}\mathcal{A}^{-1}.
\]
We have $P_{k}\circ \Xi=\chi_{k_0}|\cdot|_K^{-k_0/2}\overline{\alpha}^{k_0-k}$. From this formula, if $k=k_0$ we recover the character $\chi_{k_0}|\cdot|_K^{-k_0/2}$ we started with, and the infinity type of $\xi_k$ is $(0,-k_0)+(0,k_0-k)=(0,-k)$. Finally, the finite type of $\xi_k$ is just the product of the finite type of $\chi_{k_0}$ with the finite type of $\overline{\alpha}^{k_0-k}$, the latter being $\omega_{\overline{\mathfrak{p}}}^{k-k_0}$.
\end{proof}
\noindent As done in \cite[\S 5.2]{collins}, we can also specialize our $\Lambda$-adic Hecke character at weights $P_{k,\epsilon}=(k,\epsilon)$, for a finite order character $\epsilon:\Z_p^\times\to \overline{\Q}_p^\times$. By construction and Definition \ref{weigths} we get
\[
P_{k,\epsilon}\circ \Xi=\epsilon'\xi_k,
\]
where $\epsilon':=\epsilon \circ \alpha$. As we are interested in the case where $\chi_{k_0}$ is the unramified central character of $\mathcal{L}(f_0)$, for each $k$ the finite type of $\epsilon'\xi_k$ restricted to $\Z$ is $\epsilon'^{-1}\omega_{\overline{\mathfrak{p}}}^{k-k_0}$. Also, hereafter we suppose that the character $\epsilon$ factors through $1+p\Z_p$.
We can finally prove Proposition \ref{prop6}:

\begin{proof}[Proof of Proposition \ref{prop6}]
We recall that the Fourier-Jacobi coefficients of $\mathcal{L}^*(f)$ have been defined as
\[
g_{n/N(\mathfrak{a}),\mathfrak{a}}^*(w)=\kappa(\Omega_0)^{-1}\varepsilon(\overline{\mathfrak{a}}/\mathfrak{a})\sum_{\substack{\mathfrak{b}\in \text{Cl}_{K}\\ (\mathfrak{b},p)=1}}(\varepsilon^3\chi_\kappa^{-1})(\mathfrak{b})\kappa(\text{N}(\mathfrak{b}))^{1/2}g_{n/N(\mathfrak{a}),\mathfrak{a},\mathfrak{b}}^*(w),
\]
where
\[
g_{n/N(\mathfrak{a}),\mathfrak{a},\mathfrak{b}}^*(w)=T_{n}^*(f(\tau)\vartheta_{\overline{\mathfrak{b}}}(n\delta w,\tau))|_{\tau=-\overline{\tau}_{\mathfrak{a},\mathfrak{b}}},
\]
First of all, the Hecke equivariance follows from the continuity of the Hecke operators on the space of $p$-adic modular forms and by Theorem \ref{peigen}.

\noindent Take a sequence of modular forms $f_{i}\in S_{k_i-1}(\Gamma_1(p^rD),\omega_{K/\Q},\overline{\Z}_p)$ such that $f_{i}\to f$, $p$-adically. Up to multiplying by powers of $E_{p-1}$, the normalized Eisenstein series of weight $p-1$, we may suppose that the weights $k_i$ go to $\infty$ in the archimedean norm. Consider the sequence of lifts $\mathcal{L}(f_{i})\in \mathscr{S}(k_i,\Gamma_0^1(p^rD),\chi_{k_i},\overline{\Z}_p)$. As $p$ divides the level of the $f_i$, for $p\mid n$ the adjoint Hecke operator $T_n^*$ appearing in the $n$-th Fourier coefficient is continuous with respect to the $p$-adic topology on the space of $p$-adic modular forms. Finally, we see by the above lemma that there exists a $p$-adic analytic family of Hecke characters passing through the $\chi_{k_i}|\cdot|_K^{-k_i/2}$, so that, in particular, the limit exists and still defines a character in the same family.
\end{proof}

\noindent Let $\mathfrak{F}$ be the Hida family passing through the fixed $f_0\in S_{k_0-1}(\Gamma_1(D),\omega_{K/\Q},\overline{\Z}_p)$, as in the beginning of the section. We are in position to state and prove our main result:
\begin{Theorem}
\label{main}
For each $j\in \Z/p\Z$, the $\Lambda$-adic Fourier-Jacobi expansion $\mathfrak{L}_{j}(\mathfrak{F})$ introduced in \emph{(\ref{lambda})} defines a $\Lambda$-adic Picard modular form. More precisely, for each positive integer $k>\!\!>0$ with $k\equiv j\mod (p-1)$ and a finite character $\epsilon_{2r}:\Z_p^\times\to \overline{\Q}_p^\times$ of order $p^{2r}$ such that $$P_{k-1,\epsilon}\circ \mathfrak{F}=f_{k-1}\in S_{k-1}^{\emph{ord}}(\Gamma_1(p^{2r}D),\omega^{1-k}\epsilon_{2r}\omega_{K/\Q}),$$ we have
\[
\Omega_0^{-k}P_{k,\epsilon_{2r}}\circ\mathfrak{L}_{j}(\mathfrak{F})=\mathcal{L}^*_{k,\varepsilon,\omega_{\mathfrak{p}}^{j-1}\omega_{\overline{\mathfrak{p}}}^{k_0-1}\epsilon_{2r}\xi_k}(f_{k-1}).
\]
\end{Theorem}

\begin{proof}
The $\Lambda$-adic Fourier-Jacobi coefficients of $\mathfrak{L}_{j}(\mathfrak{F})$ have been defined by
\[
\mathfrak{g}_{n/\text{N}(\mathfrak{a}),\mathfrak{a}}^{j}(w)=\varepsilon(\overline{\mathfrak{a}}/\mathfrak{a})\sum_{\substack{\mathfrak{b}\in \text{Cl}_{K}\\ (\mathfrak{b},p)=1}}(\varepsilon^3\Xi^{-1}\omega_{\mathfrak{p}}^{j-1}\omega_{\overline{\mathfrak{p}}}^{k_0-1})(\mathfrak{b})\mathfrak{g}_{n/\text{N}(\mathfrak{a}),\mathfrak{a},\mathfrak{b}}(w),
\]
where
\[
\mathfrak{g}_{n/\text{N}(\mathfrak{a}),\mathfrak{a},\mathfrak{b}}(w)=T_{n}^*(\mathfrak{F}(\tau)\star\vartheta_{\overline{\mathfrak{b}}}(n\delta w,\tau))|_{\tau=-\overline{\tau}_{\mathfrak{a},\mathfrak{b}}},
\] 
By the definition of $\star$, we have
\[
P_{k,\epsilon}\circ\big(\mathfrak{F}(\tau)\star\vartheta_{\overline{\mathfrak{b}}}(n\delta w,\tau)\big)=\big(P_{k-1,\epsilon}\circ\mathfrak{F}(\tau)\big)\vartheta_{\overline{\mathfrak{b}}}(n\delta w,\tau)=f_{k-1}\vartheta_{\overline{\mathfrak{b}}}(n\delta w,\tau).
\]
The $T_{n}^*$ are well-defined on families, see \cite[Theorem 2.1.14]{ohta}. Finally, we apply Lemma \ref{hecke} to get $$P_{k,\epsilon}\circ\Xi^{-1}(\mathfrak{b})=(\epsilon_{2r}\xi_k)^{-1}(\mathfrak{b}),$$
which has finite type $\epsilon_{2r}\omega_{\overline{\mathfrak{p}}}^{k_0-j}$ according to our conventions on $p$-adic characters. Thus we see that, for $k\equiv j\mod (p-1)$, 
\[
(\epsilon_{2r}\xi_k)^{-1}\omega_{\mathfrak{p}}^{j-1}\omega_{\overline{\mathfrak{p}}}^{k_0-1}
\]
has finite type $\epsilon_{2r}\omega^{1-k}$, as desired.
\end{proof}

\section*{Acknowledgements}

I am deeply grateful to Professor Marc-Hubert Nicole for his uninterrupted guidance throughout the drafting of this work. I also wish to thank Maria Rosaria Pati for numerous discussions, Matteo Longo and Giovanni Rosso for their helpful comments, and finally the anonymous referee for the many suggestions which helped me improve the accuracy and legibility of this paper.

\end{document}